\documentclass[a4paper]{amsart}

\title[Conjugacy in Wreath Products and Free Solvable Groups]{The Geometry of the Conjugacy Problem in Wreath Products and Free Solvable Groups}
\author{Andrew W. Sale} 
\thanks{Email: andrew.sale@some.oxon.org. The author was supported by the EPSRC}

\usepackage{amsmath, amssymb, amsthm}
\usepackage{cite}
\usepackage{graphicx}
\usepackage{pinlabel}
\usepackage{tikz}
\usepackage{enumitem}
\usetikzlibrary{matrix,arrows,decorations.pathmorphing}

\newcommand{\Z}{\mathbb{Z}}
\newcommand{\N}{\mathbb{N}}

\newcommand{\CAT}[1]{\mathrm{CAT}(#1)}
\newcommand{\Cay}{\mathrm{Cay}}
\newcommand{\SOL}{\mathrm{SOL}}
\newcommand{\modulus}[1]{\left| #1 \right|}
\newcommand{\Supp}[1]{\mathrm{Supp}(#1)}
\newcommand{\Magnus}[1]{\left(\begin{array}{cc} \alpha(#1) & \frac{\partial^\star #1}{\partial x_1}t_1 + \ldots + \frac{\partial^\star #1}{\partial x_r}t_r \\ 0 & 1 \end{array}\right)}
\newcommand{\SolMX}[2]{{\left(\begin{array}{cc} #1 & #2 \\ 0 & 1 \end{array}\right)}}
\newcommand{\ddxi}[1]{\frac{\partial #1}{\partial x_i}}
\newcommand{\dstardxi}[1]{\frac{\partial^\star #1}{\partial x_i}}
\newcommand{\CLF}{\mathrm{CLF}}
\newcommand{\BS}[1]{\mathrm{BS}(1,#1)}
\newcommand{\maggeo}{\varphi_{\mathrm{geo}}}

\newcommand{\acknowledgements}{\vspace{3mm} \noindent\textsc{Acknowledgements: }}

\newtheorem{thm}{Theorem}[section]
\newtheorem{cor}[thm]{Corollary}
\newtheorem{lemma}[thm]{Lemma}
\newtheorem{prop}[thm]{Proposition}

\newtheorem{thmspecial}{Theorem}

\newenvironment{question}[1][Question]{\begin{trivlist}
\item[\hskip \labelsep {\bf #1}:]}{\end{trivlist}}



\begin{document}
\begin{abstract}
In this paper we describe an effective version of the conjugacy problem and study it for wreath products and free solvable groups. The problem involves estimating the length of short conjugators between two elements of the group, a notion which leads to the definition of the conjugacy length function. We show that for free solvable groups the conjugacy length function is at most cubic. For wreath products the behaviour depends on the conjugacy length function of the two groups involved, as well as subgroup distortion within the quotient group.
\end{abstract}

\maketitle

In geometric group theory there has often been a tendency to produce more effective results. For example, calculating the Dehn function of a group is an effective version of the word problem and it gives us a better understanding of its complexity. We can, furthermore, use this extra information to determine more details of the group at hand. Estimating the length of short conjugators in a group could be described as an \emph{effective} version of the conjugacy problem, and finding a control on these lengths in wreath products and free solvable group is the main motivation of this paper.

The conjugacy problem is one of Max Dehn's three decision problems for groups formulated in 1912 \cite{Dehn12}. Dehn originally described these problems because of the significance he discovered they had in the geometry of $3$--manifolds and they have since become the most fundamental problems in combinatorial and geometric group theory. Let $\Gamma$ be a finitely presented group with finite generating set $X$. The conjugacy problem asks whether there is an algorithm which determines when two given words on $X \cup X^{-1}$ represent conjugate elements in $\Gamma$. This question may also be asked of recursively presented groups, and we can try to develop our understanding further by asking whether one can find, in some sense, a short conjugator between two given conjugate elements of a group.

Suppose word-lengths in $\Gamma$, with respect to the given generating set $X$, are denoted by $\modulus{\cdot}$. The \emph{conjugacy length function} is the minimal function $\CLF_\Gamma : \N \rightarrow \N$ which satisfies the following: if $u$ is conjugate to $v$ in $\Gamma$ and $\modulus{u}+\modulus{v}\leq n$ then there exists a conjugator $\gamma \in \Gamma$ such that $\modulus{\gamma} \leq \CLF_\Gamma(n)$. One can define it more concretely to be the function which sends an integer $n$ to
				$$\max \{ \min \{ \modulus{w} : wu=vw \} : \modulus{u} + \modulus{v} \leq n \ \textrm{and $u$ is conjugate to $v$ in $\Gamma$} \}\textrm{.}$$
We know various upper bounds for the conjugacy length function in certain classes of groups. For example, Gromov--hyperbolic groups have a linear upper bound; this is demonstrated by Bridson and Haefliger \cite[Ch.III.$\Gamma$ Lemma 2.9]{BH99}. They also show that $\CAT{0}$ groups have an exponential upper bound for conjugacy length \cite[Ch.III.$\Gamma$ Theorem 1.12]{BH99}. In \cite{Sale11} we showed that some metabelian groups, namely the lamplighter groups, solvable Baumslag--Solitar groups and lattices in $\SOL$, have a linear upper bound on conjugacy length. A consequence of the result for lattices in $\SOL$, together with results of Berhstock and Dru\c{t}u \cite{BD11} and Ji, Ogle and Ramsey \cite{JOR10}, was that fundamental groups of prime $3$--manifolds have a quadratic upper bound. Behrstock and Dru\c{t}u also show that infinite order elements in mapping class groups have a linear upper bound for conjugator length, expanding on a result of Masur and Minsky \cite{MM00} in the pseudo--Anosov case. Jing Tao \cite{Tao11} developed tools to study the conjugacy of finite order elements and showed in fact that all elements in a mapping class group enjoy a linear bound.

In this paper we look at wreath products and free solvable groups. The definition of a free solvable group is as follows: let $F'=[F,F]$ denote the derived subgroup of $F$, where $F$ is the free group of rank $r$. Denote by $F^{(d)}$ the $d$--th derived subgroup, that is $F^{(d)}=[F^{(d-1)},F^{(d-1)}]$. The \emph{free solvable group} of rank $r$ and derived length $d$ is the quotient $S_{r,d}=F/F^{(d)}$. The conjugacy problem in free solvable groups was shown to be solvable by Kargapolov and Remeslennikov \cite{KR66} (see also \cite{RS70}) extending the same result for free metabelian groups by Matthews \cite{Matt66}. Recently, Vassileva \cite{Vass11} has looked at the computational complexity of algorithms to solve the conjugacy problem and the conjugacy search problem in wreath products and free solvable groups. In particular Vassileva showed that the complexity of the conjugacy search problem in free solvable groups is at most polynomial of degree $8$. Using the Magnus embedding, and in particular the fact that it its image is undistorted in the ambient wreath product, we are able to improve our understanding of the length of short conjugators in free solvable groups:

\begin{thmspecial}\label{thmspecial:clf free solvable}
Let $r,d > 1$. Then the conjugacy length function of the free solvable group $S_{r,d}$ is bounded above by a cubic polynomial.
\end{thmspecial}

The Magnus embedding allows us to see $S_{r,d+1}$ as a subgroup of $\Z^r \wr S_{r,d}$. Hence, if we are to use the Magnus embedding, we need to understand conjugacy in wreath products. For such groups the conjugacy problem was studied by Matthews \cite{Matt66}, who showed that for two recursively presented groups $A,B$ with solvable conjugacy problem, their wreath product $A\wr B$ has solvable conjugacy problem if and only if $B$ has solvable power problem. In Section \ref{sec:conjugacy in wreath products} we show that for such $A$ and $B$ there is an upper bound for the conjugacy length function of $A \wr B$ which depends on the conjugacy length functions of $A$ and $B$ and on the subgroup distortion of infinite cyclic subgroups of $B$ (note that these distortion functions are related to the power problem). In the case when the $B$--component is of infinite order we do not need the conjugacy length function of $A$.

\begin{thmspecial}\label{thmspecial:conj length in wreath products}
Suppose $A$ and $B$ are finitely generated groups. Let $u=(f,b),v=(g,c)$ be elements in $\Gamma=A\wr B$. Then $u,v$ are conjugate if and only if there exists a conjugator $\gamma \in \Gamma$ such that 
\begin{equation*}\begin{array}{ll}
d_\Gamma (1,\gamma) \leq (n+1)P(2\delta_{\langle b \rangle}^B(P)+1) &  \textrm{if $b$ is of infinite order, or}\\
d_\Gamma (1,\gamma) \leq  P(N+1)(2n+\CLF_A(n)+1) &  \textrm{if $b$ is of finite order $N$,}
\end{array}\end{equation*}
where  $n = d_\Gamma(1,u)+d_\Gamma(1,v)$, $\delta_H^B$ is the subgroup distortion function of $H < B$ and $P=2n$ if $(f,b)$ is not conjugate to $(1,b)$ and $P=n+\CLF_B(n)$ otherwise.
\end{thmspecial}

Note that when we avoid the conjugacy classes which contain elements of the form $(1,b)$, and if the $B$--component is also of infinite order, then we can find a short conjugator whose length is bounded above by $2(n^2+n)(2\delta_{\langle b \rangle}^B(2n)+1)$. This appears to be independent of the conjugacy length functions of both $A$ and $B$, but such a statement seems counter-intuitive. It suggests that the conjugacy length functions are somehow wrapped up in the word length of the wreath product.


In order to apply Theorem \ref{thmspecial:conj length in wreath products} to free solvable groups we must understand the distortion of their cyclic subgroups. Given a finitely generated subgroup $H$ in a finitely generated group $G$, with corresponding word metrics $d_H$ and $d_G$ respectively, the \emph{subgroup distortion function} $\delta_H^G$ compares the size of an element in a Cayley graph of $H$ with its size in a Cayley graph of $G$. It is defined as
				$$\delta_H^G(n)=\max \{ d_{H}(e_G,h)\mid d_G(e_G,h)\leq n \}\textrm{.}$$
Subgroup distortion is studied up to an equivalence relation of functions. For functions $f,g:\N \rightarrow [0,\infty)$ we write $f \preceq g$ if there exists an integer $C>0$ such that $f(n) \leq Cg(Cn)$ for all $n \in \N$. The two functions are equivalent if both $f \preceq g$ and $g \preceq f$. In this case we write $f \asymp g$. Up to this equivalence we can talk about \emph{the} distortion function for a group. If the distortion function of a subgroup $H$ satisfies $\delta_H^G(n) \asymp n$ then we say $H$ is undistorted in $G$, otherwise $H$ is distorted.

We also investigate lower bounds for the conjugacy length function of wreath products. The distortion of cyclic subgroups of $B$ plays an important role here too, though it is not the only tool we use. In particular, when considering wreath products of the form $A\wr B$ when $B$ contains a copy of $\Z^2$ we make use of the fact that right-angled triangles in $\Z^2$ contain an area which is quadratic with respect to the perimeter length. This is used to give a quadratic lower bound on the conjugacy length function of these wreath products.

\begin{thmspecial}\label{thmspecial:wreath CLF lower bounds}
Let $A$ and $B$ be finitely generated groups. If $B$ is not virtually cyclic then for any $x$ of infinite order in the centre of $B$ we have:
		$$\CLF_{A \wr B}(n) \succeq \delta_{\langle x \rangle}^B(n).$$
If $B$ contains a copy of $\Z^2$ then:
		$$\CLF_{A\wr B}(n)\succeq n^2.$$
In particular for some $\alpha \in [2,3]$ we have:
		$$\CLF_{A\wr \Z^r}(n) \asymp n^\alpha.$$
\end{thmspecial}

Osin \cite{Osin01} has described the distortion functions of subgroups of finitely generated nilpotent groups. In particular, for a $c$--step nilpotent group $N$ his result implies that the maximal distortion of a cyclic subgroup of $N$ will be $n^c$. A consequence of Theorem \ref{thmspecial:conj length in wreath products}, Theorem \ref{thmspecial:wreath CLF lower bounds} and Osin's work is that when restricting to elements in $A \wr N$ not conjugate to an element of the form $(1,b)$, the conjugacy length function will be $n^\alpha$, where $\alpha \in [c,c+2]$. However, since we do not yet know the conjugacy length function of a general $c$--step nilpotent group, apart from this lower bound we cannot estimate the conjugacy length function of $A \wr N$. In the particular case when $N$ is $2$--step nilpotent we know its conjugacy length function is quadratic by Ji, Ogle and Ramsey \cite{JOR10}. This implies that the conjugacy length function of $A\wr N$, when $N$ is torsion-free, is $n^\alpha$ for some $\alpha \in [2,7]$.

\vspace{3mm}

The paper is divided into two sections, the first dealing with wreath products and the second with free solvable groups. Within Section \ref{sec:conjugacy in wreath products}, we obtain the upper bound in Section \ref{sec:short conj in wreath - upper bound} with the lower bounds discussed in Section \ref{subsec:lower bounds}. 
We begin Section \ref{sec:free solvable} by recapping some details of the Magnus embedding. The subgroup distortion of cyclic subgroups of $S_{r,d}$ is dealt with in Section \ref{sec:subgroup distortion}, before we move onto conjugacy in free solvable groups in the final section.

\acknowledgements The author would like to thank Cornelia Dru\c{t}u for many valuable discussions on this paper. Alexander Olshanskii's comments on a draft copy were also very helpful, as were discussions with Romain Tessera. He would also like to thank Ralph St{\"o}hr for providing a reference for Lemma \ref{lem:kernel of alpha^star}.

\section{Wreath Products}\label{sec:conjugacy in wreath products}

\subsection{Geometry of wreath products}\label{sec:wreath word length}

Let $A,B$ be groups. Denote by $A^{(B)}$ the set of all functions from $B$ to $A$ with finite support, and equip it with pointwise multiplication to make it a group. The (restricted) wreath product $A \wr B$ is the semidirect product $A^{(B)} \rtimes B$. To be more precise, the elements of $A\wr B$ are pairs $(f,b)$ where $f \in A^{(B)}$ and $b \in B$. Multiplication in $A\wr B$ is given by
				$$(f,b)(g,c)=(fg^b,bc), \ \ \ f,g \in A^{(B)} , \  b,c, \in B$$
where $g^b(x)=g(b^{-1}x)$ for each $x \in B$. The identity element in $B$ will be denoted by $e_B$, while we use $1$ to denote the trivial function from $B$ to $A$.

We can paint a picture of $A \wr B$ in a similar vein to the well-known picture for lamplighter groups $\Z_q \wr \Z$. In the more general context where we consider $A \wr B$, the problem of determining the length of an element requires a solution to the travelling salesman problem on a Cayley graph $\Cay(B,X)$ of $B$, with respect to some finite generating set $X$.  Suppose we take an element $(f,b) \in A \wr B$. We can think of this as a set of instructions given to a salesman, who starts the day at the vertex in $\Cay(B,X)$ labelled by the identity. The instructions comprise
\begin{itemize}
\item a list of vertices to visit (the support $\Supp{f}$);
\item a particular element of $A$ to ``sell'' at each of these vertices (determined by the image of $f$ at each vertex); and
\item a final vertex $b$ where the salesman should end the day.
\end{itemize}
Intuitively, therefore, we would expect the word length of $(f,b)$ to be the ``quickest'' way to do this. In particular, the salesman needs to find the shortest route from the identity vertex to $b$ in which every vertex of $\Supp{f}$ is visited at least once. We will denote the length of such a path by $K(\Supp{f},b)$, following the notation of \cite{deCo06}.
 
The following Lemma formalises this idea. A proof of the Lemma, for a slightly more general context, can be found in the Appendix of \cite[Lemma A.1]{deCo06} and also in \cite[Theorem 3.4]{DO11}. We fix a finite generating set $X$ for $B$, let $S=X\cup X^{-1}$, and for each $b \in B$ denote the corresponding word-length as $\modulus{b}$. We consider the left-invariant word metric on $B$, given by $d_B(x,y):=\modulus{x^{-1}y}$. Similarly, fix a finite generating set $T$ for $A$ and let $\modulus{\cdot}$ denote the word-length. For $f \in A^{(B)}$, let
				$$\modulus{f} = \sum_{x \in B} \modulus{f(x)}\textrm{.}$$
Let $A_{e_B}$ be the subgroup of $A^{(B)}$ consisting of those elements whose support is contained in $\{ e_B \}$. Then $A_{e_B}$ is generated by $\{f_t \mid t \in T \}$ where $f_t(e_B)=t$ for each $t \in T$ and $\Gamma$ is generated by $\{ (1,s),(f_t,e_B) \mid s \in S,t\in T \}$. With respect to this generating set, we will let $\modulus{(f,b)}$ denote the corresponding word-length for $(f,b) \in \Gamma$.
 
\begin{lemma}[{\hspace{-0.2mm}}{\cite[Lemma A.1]{deCo06}}]\label{lem:wreath metric}
Let $(f,b) \in \Gamma = A \wr B$, where $A,B$ are finitely generated groups. Then
$$
\modulus{(f,b)} = K(\Supp{f},b) + \modulus{f}$$
where $K(\Supp{f},b)$ is the length of the shortest path in the Cayley graph $\Cay(B,S)$ of $B$ from $e_B$ to $b$, travelling via every point in $\Supp{f}$.
\end{lemma}

\subsection{Conjugacy in wreath products}

Let $A$ and $B$ be finitely generated groups. By a result of Matthews \cite{Matt66}, when $A$ and $B$ are recursively presented with solvable conjugacy problem and when $B$ also has solvable power problem, the group $\Gamma = A \wr B$ has solvable conjugacy problem. In what follows we will not need these assumptions, we will only assume that $A$ and $B$ are finitely generated.

Fix $b \in B$ and let $\{t_i \mid i \in I \}$ be a set of right-coset representatives for $\langle b \rangle$ in $B$. We associate to this a family of maps $\pi_{t_i}^{(z)}:A^{(B)} \rightarrow A$ for each $z$ in $B$ as follows:
$$
\pi_{t_i}^{(z)}(f)=\left\{ \begin{array}{cc} \prod\limits_{j=0}^{N-1} f(z^{-1}b^jt_i) & \textrm{for} \ b \  \textrm{of finite order} \ N \\ \prod\limits_{j=-\infty}^{\infty} f(z^{-1}b^jt_i) & \textrm{for} \ b \ \textrm{of infinite order.}\end{array}\right.$$
The products above are taken so that the order of multiplication is such that $f(t_ib^jz^{-1})$ is to the left of $f(t_ib^{j-1}z^{-1})$ for each $j$. When $z=e_B$ we denote $\pi_{t_i}^{(z)}$ by $\pi_{t_i}$.

\begin{prop}[Matthews \cite{Matt66}]\label{prop:Matthews-conjugacy in wreath products}
Fix a family $\{ t_i \mid i \in I\}$ of right-coset representatives for $\langle b \rangle$ in $B$. Two elements $(f,b)$ and $(g,c)$ are conjugate in $A \wr B$ if and only if there exists an element $z$ in $B$ such that $bz = zc$ and for all $i\in I$ either
\begin{itemize}
\item $\pi_{t_i}^{(z)}(g)=\pi_{t_i}(f)$ if $b$ is of infinite order; or
\item $\pi_{t_i}^{(z)}(g)$ is conjugate to $\pi_{t_i}(f)$ if $b$ is of finite order.
\end{itemize}
For such $z$ in $B$, a corresponding function $h$ such that $(f,b)(h,z)=(h,z)(g,c)$ is defined as follows: if $b$ is of infinite order then for each $i \in I$ and each $k \in \Z$ we set
				$$h(b^kt_i) = \left(\prod\limits_{j\leq k} f(b^jt_i)\right)\left( \prod\limits_{j\leq k} g(z^{-1}b^jt_i)\right)^{-1}$$ or if $b$ is of finite order $N$, then for each $i \in I$ and each $k = 0, \ldots , N-1$ we set
				$$h(b^kt_i) = \left(\prod\limits_{j=0}^{k} f(b^jt_i)\right) \alpha_{t_i} \left( \prod\limits_{j=0}^{k} g(z^{-1}b^jt_i)\right)^{-1}$$ where $\alpha_{t_i}$ is any element satisfying $\pi_{t_i}(f) \alpha_{t_i} = \alpha_{t_i} \pi_{t_i}^{(z)}(g)$.
\end{prop}

\subsection{Upper Bounds for Lengths of Short Conjugators}\label{sec:short conj in wreath - upper bound}
Proposition \ref{prop:Matthews-conjugacy in wreath products} gives us an explicit description of a particular conjugator for two elements in $A \wr B$. The following Lemma tells us that any conjugator between two elements has a concrete description similar to that given by Matthews in the preceding Proposition. With this description at our disposal we will be able to determine their size and thus find a short conjugator.

\begin{lemma}\label{lem:coset reps}
Let $(h,z),(f,b),(g,c) \in A \wr B$ be such that $(f,b)(h,z)=(h,z)(g,c)$. Then there is a set of right-coset representatives $\{t_i \mid i \in I \}$ of $\langle b \rangle$ in $B$ such that, if $b$ is of infinite order then
				$$h(b^kt_i) = \left(\prod\limits_{j\leq k} f(b^jt_i)\right)\left( \prod\limits_{j\leq k} g(z^{-1}b^jt_i)\right)^{-1}$$
for every $i \in I$ and $k \in \Z$; if $b$ is of finite order $N$ then
				$$h(b^kt_i) = \left(\prod\limits_{j=0}^{k} f(b^jt_i)\right) \alpha_{t_i} \left( \prod\limits_{j=0}^{k} g(z^{-1}b^jt_i)\right)^{-1}$$
for every $i \in I$ and $k=0,\ldots,N-1$ and where $\alpha_{t_i}$ satisfies $\pi_{t_i}(f) \alpha_{t_i} = \alpha_{t_i} \pi_{t_i}^{(z)}(g)$. Furthermore, for any element $\alpha_{t_i}$ satisfying this relationship there exists some conjugator $(h,z)$ with $h$ of the above form.
\end{lemma}

\begin{proof}
Fix a set of coset representatives $\{s_i \mid i \in I \}$. By Matthews' argument there exists a conjugator $(h_1,z_1) \in A \wr B$ for $(f,b)$ and $(g,c)$ as described in \mbox{Proposition \ref{prop:Matthews-conjugacy in wreath products},} with respect to the coset representatives $\{s_i \mid i \in I\}$. Since $(h,z)$ and $(h_1,z_1)$ are both conjugators, it follows that there exists some $(\psi,y)$ in $Z_\Gamma(f,b)$ such that $(h,z) = (\psi,y)(h_1,z_1)$.
This tells us that $z=y z_1$ and also that $h(x) = \psi(x) h_1(y^{-1}x)$ for each $x \in B$. Since $(\psi,y)$ is in the centraliser of $(f,b)$, we obtain two identities:
\begin{eqnarray}
yb &=& by \\
\label{eqn:lem:coset reps - centraliser equation 2}\psi(x)f(y^{-1}x) &=& f(x)\psi(b^{-1}x) \ \ \ \forall x \in B\textrm{.}
\end{eqnarray}
For each $i \in I$ we set $t_i = y s_i$. First suppose that $b$ is of infinite order. Then
\begin{eqnarray*}
h(b^kt_i) & = & \psi(b^k t_i)h_1(y^{-1} b^k t_i)\\
		  & = & \psi(b^k t_i)h_1(b^k s_i)  \\
		  & = & \psi(b^k t_i) \left(\prod\limits_{j \leq k}f(b^js_i)\right)\left(\prod\limits_{j \leq k} g(z_1^{-1}b^js_i) \right)^{-1}\\
		  & = & \psi(b^k t_i) \left(\prod\limits_{j \leq k}f(y^{-1}b^jt_i)\right)\left(\prod\limits_{j \leq k} g(z^{-1}b^jt_i) \right)^{-1}\textrm{.}
\end{eqnarray*}
We can apply equation (\ref{eqn:lem:coset reps - centraliser equation 2}) once, and then repeat this process to shuffle the $\psi$ term past all the terms involving $f$. This process terminates and the $\psi$ term vanishes because of the finiteness of support of both $\psi$ and of $f$. Hence, as required, we obtain:
\begin{eqnarray*}
h(b^kt_i) &=&  f(b^kt_i) \psi(b^{k-1} t_i) \left(\prod\limits_{j \leq k-1}f(y^{-1} b^j t_i)\right)  \left( \prod\limits_{j \leq k} g(z^{-1}b^jt_i) \right)^{-1}\\
& \vdots & \\
& = & \left( \prod\limits_{j \leq k} f(b^jt_i)\right)\left( \prod\limits_{j \leq k} g(z^{-1}b^jt_i) \right)^{-1} \textrm{.}
\end{eqnarray*}

If instead $b$ is of finite order, $N$ say, then for $0 \leq k \leq N-1$ we obtain
				$$h(b^kt_i) = \psi(b^kt_i) \left(\prod\limits_{j=0}^k f(y^{-1}b^jt_i) \right) \alpha_{s_i} \left( \prod\limits_{j=0}^k g(z^{-1}b^jt_i)\right)^{-1}$$
where $\alpha_{s_i}$ can be chosen, by Proposition \ref{prop:Matthews-conjugacy in wreath products}, to be any element which satisfies $\pi_{s_i}(f) \alpha_{s_i} = \alpha_{s_i} \pi_{s_i}^{(z_1)}(g)$. With equation (\ref{eqn:lem:coset reps - centraliser equation 2}) the $\psi (b^kt_i)$ term can be shuffled past the terms involving $f$. Unlike in the infinite order case the $\psi$ term will not vanish:
				$$h(b^kt_i) = \left(\prod\limits_{j=0}^k f(b^jt_i) \right) \psi(b^{-1}t_i) \alpha_{s_i} \left( \prod\limits_{j=0}^k g(z^{-1}b^jt_i)\right)^{-1} \textrm{.}$$
To confirm that $h$ is of the required form, all that is left to do is to verify that if we set $\alpha_{t_i}=\psi(b^{-1}t_i) \alpha_{s_i}$ then it will satisfy $\pi_{t_i}(f) \alpha_{t_i} = \alpha_{t_i} \pi_{t_i}^{(z)}(g)$. We will prove this while proving the final statement of the Lemma: that any element $\alpha_{t_i}$ satisfying $\pi_{t_i}(f) \alpha_{t_i} = \alpha_{t_i} \pi_{t_i}^{(z)}(g)$ will appear in this expression for some conjugator between $(f,b)$ and $(g,c)$.  Set 
		$$\mathcal{C}_{t_i}=\{ \alpha \mid \pi_{t_i}(f)\alpha = \alpha\pi_{t_i}^{(z)}(g)\} \ \textrm{ and } \ \mathcal{C}_{s_i}=\{ \alpha \mid \pi_{s_i}(f)\alpha = \alpha\pi_{s_i}^{(z_1)}(g)\}.$$
By Proposition \ref{prop:Matthews-conjugacy in wreath products}, we can choose $h_1$ above so that any element of $\mathcal{C}_{s_i}$ appears above in the place of $\alpha_{s_i}$. We need to check that $\mathcal{C}_{t_i}=\psi(b^{-1}t_i)\mathcal{C}_{s_i}$. Observe that we have two equalities:
\begin{eqnarray}
\label{eqn:lem:coset reps - equality 1}\psi(b^{-1}t_i)\pi_{s_i}(f)&=&\pi_{t_i}(f)\psi(b^{-1}t_i)\\
\label{eqn:lem:coset reps - equality 2}\pi_{t_i}^{(z)}(g)&=&\pi_{s_i}^{(z_1)}(g)
\end{eqnarray}
Equation (\ref{eqn:lem:coset reps - equality 2}) is straight-forward to show and was used above in the infinite order argument, while equation (\ref{eqn:lem:coset reps - equality 1}) follows by applying equation (\ref{eqn:lem:coset reps - centraliser equation 2}) $N$ times:
		$$\psi(b^{-1}t_i)\prod\limits_{j=0}^{N-1}f(b^js_i)=\prod\limits_{j=0}^{N-1}f(yb^js_i)\psi(b^{-(N+1)}t_i)$$
and then using the facts that $b$ has order $N$ and $y$ is in the centraliser of $b$.

Suppose that $\alpha_{s_i} \in \mathcal{C}_{s_i}$. Then, using equations (\ref{eqn:lem:coset reps - equality 1}) and (\ref{eqn:lem:coset reps - equality 2}):
\begin{eqnarray*}
e_A &=& \alpha_{s_i}^{-1} \pi_{s_i}(f)^{-1} \alpha_{s_i} \pi_{s_i}^{(z_1)}(g)\\
	&=& \alpha_{s_i}^{-1} \psi(b^{-1}t_i)^{-1} \pi_{t_i}(f)^{-1}  \psi(b^{-1}t_i) \alpha_{s_i} \pi_{t_i}^{(z)}(g).
\end{eqnarray*}
This confirms that $\psi(b^{-1}t_i) \mathcal{C}_{s_i} \subseteq \mathcal{C}_{t_i}$. On the other hand, suppose instead that $\alpha_{t_i} \in \mathcal{C}_{t_i}$. Then
\begin{eqnarray*}
e_A &=& \alpha_{t_i}^{-1} \pi_{t_i}(f)^{-1} \alpha_{t_i} \pi_{t_i}^{(z)}(g)\\
	&=& \alpha_{t_i}^{-1} \psi(b^{-1}t_i) \pi_{s_i}(f)^{-1}  \psi(b^{-1}t_i)^{-1} \alpha_{t_i} \pi_{s_i}^{(z_1)}(g).
\end{eqnarray*}
Hence $\psi(b^{-1}t_i)^{-1}\alpha_{t_i} \in \mathcal{C}_{s_i}$. In particular we get $\mathcal{C}_{t_i}=\psi(t^{-1})\mathcal{C}_{s_i}$ as required.
\end{proof}

Obtaining a short conjugator will require two steps. Lemma \ref{lem:wreath conjugator with short B part} is the first of these steps. Here we actually find the short conjugator, while in Lemma \ref{lem:wreath conj bound depending on B part} we show that the size of a conjugator $(h,z)$ can be bounded by a function involving the size of $z$ but independent of $h$ altogether.

Recall that the conjugacy length function of $B$ is the minimal function $$\CLF_B:\N \rightarrow \N$$ such that  if $b$ is conjugate to $c$ in $B$ and $d_B(e_B,b)+d_B(e_B,c)\leq n$ then there exists a conjugator $z \in B$ such that $d_B(e_B,z) \leq \CLF_B(n)$.

\begin{lemma}\label{lem:wreath conjugator with short B part}
Suppose $u=(f,b),v=(g,c)$ are conjugate elements in $\Gamma = A \wr B$ and let $n=d_\Gamma(1,u)+d_\Gamma(1,v)$. Then there exists $\gamma=(h,z) \in \Gamma$ such that $u \gamma = \gamma v$ and either:
\begin{enumerate}[label=(\arabic{enumi})]
\item\label{item:lem:wreath conjugator with short B part:conjugate to (1,b)} $d_B(e_B,z) \leq \CLF_B(n)$ if $(f,b)$ is conjugate to $(1,b)$; or
\item\label{item:lem:wreath conjugator with short B part:not conjugate to (1,b)} $d_B(e_B,z) \leq n$ if $(f,b)$ is not conjugate to $(1,b)$.
\end{enumerate}
\end{lemma}

\begin{proof}
Without loss of generality we may assume that $d_\Gamma(1,u) \leq d_\Gamma(1,v)$. By Lemma \ref{lem:coset reps}, if $(h_0,z_0)$ is a conjugator for $u$ and $v$ then there exists a family of right-coset representatives $\{ t_i  \mid i \in I \}$ for $\langle b \rangle$ in $B$ such that 
	$$ \pi_{t_i}^{(z_0)}(g)=\pi_{t_i}(f) \ \ \ \textrm{or} \ \ \ \pi_{t_i}^{(z_0)}(g) \ \textrm{is conjugate to} \ \pi_{t_i}(f)$$
for every $i \in I$ according to whether $b$ is of infinite or finite order respectively (the former follows from the finiteness of the support of the function $h$ given by Lemma \ref{lem:coset reps}). 

By Proposition \ref{prop:Matthews-conjugacy in wreath products}, $(f,b)$ is conjugate to $(1,b)$ if and only if $\pi_{t_i}(f)=e_A$ for every $i \in I$. In this case we take 
\begin{equation*}
				h(b^kt_i)=  \left\{ \begin{array}{ll} \prod\limits_{j\leq k} f(b^jt_i) & \textrm{if $b$ is of infinite order;}\\
				\prod\limits_{j=0}^k f(b^jt_i) & \textrm{if $b$ is of finite order $N$ and $0 \leq k < N$.}\end{array}\right.
\end{equation*}
One can then verify that $(f,b)(h,e_B)=(h,e_B)(1,b)$. Thus we have reduced \ref{item:lem:wreath conjugator with short B part:conjugate to (1,b)} to the case when $u=(1,b)$ and $v=(1,c)$. For this we observe that any conjugator $z$ for $b,c$ in $B$ will give a conjugator $(1,z)$ for $u,v$ in $A \wr B$. Thus \ref{item:lem:wreath conjugator with short B part:conjugate to (1,b)} follows.

If on the other hand $(f,b)$ is not conjugate to $(1,b)$ then by Proposition \ref{prop:Matthews-conjugacy in wreath products}, $\pi_{t_i}(f) \neq e_A$ for some $i \in I$. Fix some such $i$, observe that there exists \mbox{$k \in \Z$} satisfying $b^kt_i \in \Supp{f}$ and there must also exist some $j \in \Z$ so that \mbox{$z_0^{-1}b^jt_i \in \Supp{g}$.} Pre-multiply $(h_0,z_0)$ by $(f,b)^{k-j}$ to get $\gamma=(h,z)$, where \mbox{$z=b^{k-j}z_0$} and $\gamma$ is a conjugator for $u$ and $v$ since $(f,b)^{k-j}$ belongs to the centraliser of $u$ in $\Gamma$. By construction, $z^{-1} b^k t_i=z_0^{-1} b^j t_i$ and hence is contained in the support of $g$. We finish by applying the triangle inequality and using the left-invariance of the word metric $d_B$ as follows:
\begin{eqnarray*}
d_B(e_B,z^{-1}) & \leq & d_B(e_B,z^{-1}b^kt_i) + d_B(z^{-1}b^kt_i,z^{-1}) \\
		& \leq & d_B(e_B,b^kt_i) + d_B(e_B, z^{-1}b^kt_i) \\
		& \leq & K(\Supp{f},b) + K(\Supp{g},c) \\
		& \leq &  d_\Gamma(1,u)+d_\Gamma(1,v).
\end{eqnarray*}
This completes the proof.
\end{proof}

Soon we will give Theorem \ref{thm:wreath conj length}, which will describe the length of short conjugators in wreath products $A \wr B$ where $B$ is torsion-free. Before we dive into this however, it will prove useful in Section \ref{sec:Conjugacy in free solvable groups}, when we look at conjugacy in free solvable groups, to understand how the conjugators are constructed. In particular, it is important to understand that the size of a conjugator $(h,z) \in A \wr B$ can be expressed in terms of the size of $z$ in $B$ with no need to refer to the function $h$ at all. This is what we explain in Lemma \ref{lem:wreath conj bound depending on B part}.

For $b \in B$, let $\delta_{\langle b \rangle}^B(n) = \max \{ m \in \Z \mid d_B(e_B,b^m) \leq n \}$ be the subgroup distortion of $\langle b \rangle$ in $B$. Fix a finite generating set $X$ for $B$ and let $\Cay(B,X)$ be the corresponding Cayley graph.

\begin{lemma}\label{lem:wreath conj bound depending on B part}
Suppose $u=(f,b),v=(g,c)$ are conjugate elements in $\Gamma = A \wr B$ and let $n=d_\Gamma(1,u)+d_\Gamma(1,v)$. Suppose also that $b$ and $c$ are of infinite order in $B$. If $\gamma = (h,z)$ is a conjugator for $u$ and $v$ in $\Gamma$ then 
\begin{samepage}				$$d_{\Gamma}(1,\gamma) \leq (n+1)P(2\delta_{\langle b \rangle}^B(P)+1)$$
where $P=d_B(1,z)+n$.\end{samepage}
\end{lemma}

\begin{proof}
Without loss of generality we may assume $d_\Gamma(1,u) \leq d_\Gamma(1,v)$. From Lemma \ref{lem:coset reps} we have an explicit expression for $h$. We use this expression to give an upper bound for the size of $(h,z)$, making use of Lemma \ref{lem:wreath metric} which tells us
				$$d_{\Gamma}(1,\gamma) = K(\Supp{h},z) + \modulus{h}$$
where $K(\Supp{h},z)$ is the length of the shortest path in $\Cay(B,X)$ from $e_B$ to $z$ travelling via every point in $\Supp{h}$ and $\modulus{h}$ is the sum of terms $d_A(e_A,f(x))$ over all $x \in B$.

\begin{figure}[t!]
\labellist \small\hair 4pt
	\pinlabel $e_B$ [b] at 156 242
	\pinlabel $z$ [t] at 192 2
	\pinlabel $q_0$ [b] at 112 235
	\pinlabel $p_1$ [b] at 164 218
	\pinlabel $q_s$ [l] at 224 29
	\pinlabel $p_s$ [t] at 206 58
	\pinlabel $\langle b\rangle t_1$ [l] at 359 236
	\pinlabel $\langle b\rangle t_2$ [l] at 359 212
	\pinlabel $\langle b\rangle t_{s-1}$ [l] at 359 61
	\pinlabel $\langle b\rangle t_s$ [l] at 361 30
\endlabellist
\vspace{1mm}\centering\includegraphics[width=9cm]{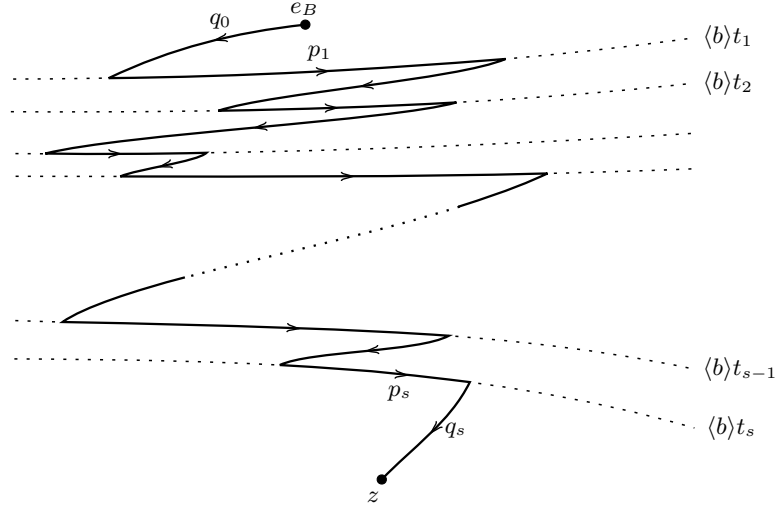}\vspace{3mm}
\caption[A short path covering $\Supp{h}$]{We build a path from $e_B$ to $b$ by piecing together paths $q_i$ and $p_i$, where the paths $p_i$ run though the intersection of $\Supp{h}$ with a coset $\langle b \rangle t_i$ and the paths $q_i$ connect these cosets.}\label{fig:K(Supp(h),z)}
\end{figure}

We begin by obtaining an upper bound on the size of $K(\Supp{h},z)$. To do this we build a path from $e_B$ to $z$, zig-zagging along cosets of $\langle b \rangle$, see Figure \ref{fig:K(Supp(h),z)}. Lemma \ref{lem:coset reps} tells us that there is a family of right-coset representatives $\{ t_i \}_{i \in I}$ such that
				$$h(b^kt_i) = \left(\prod\limits_{j \leq k}f(b^jt_i)\right) \left( \prod\limits_{j \leq k} g(z^{-1}b^jt_i) \right)^{-1}$$
for every $i \in I$ and $k \in \Z$. This expression for $h$ tells us where in each coset the support of $h$ will lie. In particular, note that if we set $C=\Supp{f}\cup z\Supp{g}$, then 
		$$\Supp{h}\cap \langle b \rangle t_i \neq \emptyset\  \Longrightarrow\  C \cap \langle b \rangle t_i \neq \emptyset.$$
Furthermore, in each coset the support of $h$ must lie between some pair of elements in $C$. Let $t_1 , \ldots , t_s$ be all the coset representatives for which $\Supp{h}$ intersects the coset $\langle b \rangle t_i$. The number $s$ of such cosets is bounded above by the size of the set $\Supp{f} \cup \Supp{g}$, which is bounded above by $d_\Gamma(1,u)+d_\Gamma(1,v)=n$.

If we restrict our attention to one of these cosets, $\langle b \rangle t_i$, then there exist integers $m_1 < m_2$ such that $b^j t_i \in \Supp{h}$ implies $m_1 \leq j \leq m_2$. We can choose $m_1$ and $m_2$ so that $b^mt_i \in C$ for $m \in \{m_1,m_2\}$. Let $p_i$ be a piecewise geodesic in the Cayley graph of $B$ which connects $b^{m_1}t_i$ to $b^{m_2}t_i$ via $b^jt_i$ for every $m_1 < j < m_2$. The length of $p_i$ will be at most 
				$$d_B(b^jt_i,b^{j+1}t_i)\delta_{\langle b \rangle}^B(\mathrm{diam}(C))$$
for any $j \in \Z$. Choose $j\in \Z$ such that $b^jt_i \in C$. In that case that $b^jt_i \in \Supp{f}$ we get that
\begin{eqnarray*}
 d_B(b^j t_i , b^{j+1}t_i) & \leq & d_B(b^j t_i , b) + d_B(b,b^{j+1} t_i) \\
			   & =    & d_B(b^j t_i, b) + d_B(e_B,b^j t_i) \\
			   & \leq & K(\Supp{f},b) \leq n
\end{eqnarray*}
where the last line follows because any path from $e_B$ to $b$ via all points in $\Supp{f}$ will have to be at least as long as the path from $e_B$ to $b$ via the point $b^jt_i$. Similarly, in the case when $z^{-1}b^jt_i \in \Supp{g}$, we get
\begin{eqnarray*}
 d_B(b^j t_i, b^{j+1} t_i) & \leq & d_B(b^jt_i,zc)+d_B(zc,b^{j+1}t_i)\\
 						   & = & d_B(b^jt_i,zc) + d_B(z,b^jt_i)\\
 						   & \leq & K(\Supp{g},c) \leq n
\end{eqnarray*}
where we obtain the last line because a shortest path from $z$ to $zc$ via $z\Supp{g}$ will have length precisely $K(\Supp{g},c)$. Hence, in either case we get that the path $p_i$ has length bounded above by $n\delta_{\langle b \rangle}^B(\mathrm{diam}(C))$.

We will now show that $\mathrm{diam}(C \cup \{e_B,z\})\leq n + d_B(1,z)=P$. This diameter will be given by the length of a path connecting some pair of points in this set. We take a path through $e_B$, $z$ and all points in the set $C$, a path such as that in Figure \ref{fig:supp(h) path}. The length of this path will certainly be bigger than the diameter. Hence we have
\begin{eqnarray*}
\mathrm{diam}(C) & \leq & K(\Supp{f},b)+ d_B(1,z) + K(\Supp{g},c) \\
  & \leq & n+d_B(1,z)=P.
\end{eqnarray*}

\begin{figure}[t!] \vspace{3mm}
\labellist \small\hair 4pt
	\pinlabel $e_B$ [r] at 1 1
	\pinlabel $z$ [r] at 8 174
	\pinlabel $\Supp{f}$ [t] at 81 2
	\pinlabel $z\Supp{g}$ [b] at 93 231
	\pinlabel $b$ [l] at 182 53
	\pinlabel $bz=zc$ [l] at 176 190
\endlabellist
\vspace{1mm}\centering\includegraphics[width=3.5cm]{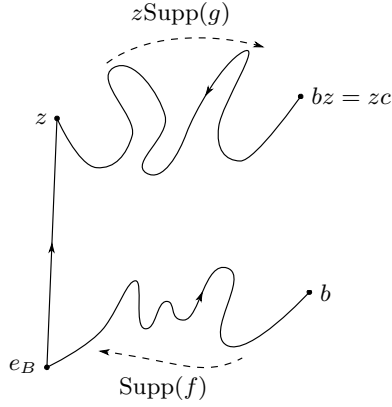}\vspace{1.5mm}
\caption[Bounding $D$]{The concatenation of three paths: a path from $b$ to $e_B$ through $\Supp{f}$, followed by a path from $e_B$ to $z$, then finish off with a path from $z$ to $zc$ through each point in $z\Supp{g}$.}\label{fig:supp(h) path}
\end{figure}

For $i = 1 , \ldots , s-1$ let $q_i$ be a geodesic path which connects the end of $p_i$ with the start of $p_{i+1}$. Let $q_0$ connect $e_B$ with the start of $p_1$ and $q_s$ connect the end of $p_s$ with $z$. Then the concatenation of paths $q_0, p_1, q_1 , \ldots , q_{s-1},p_s,q_s$ is a path from $e_B$ to $z$ via every point in $\Supp{h}$.

For each $i$, the path $q_i$ will be a geodesic connecting two points of $C \cup \{ e_B,z \}$. The above upper bound for the diameter of this set tells us that each $q_i$ will have length at most $n + d_B(1,z)=P$.

Hence our path $q_0, p_1, q_1 , \ldots , q_{s-1},p_s,q_s$ has length bounded above by
				$$(n+1)\textit{P} + n^2\delta_{\langle b \rangle}^B(P)\leq (n+1)P(\delta_{\langle b \rangle}^B(P)+1)$$
thus giving an upper bound for $K(\Supp{h},z)$.

Now we need to turn our attention to an upper bound for $\modulus{h}$. By the value of $h(b^kt_i)$ given to us by Lemma \ref{lem:coset reps} we see that
\begin{eqnarray*}
d_A(e_A,h(b^kt_i)) & \leq & \sum_{j \leq k} d_A(e_A,g(z^{-1}b^jt_i)) + \sum_{j \leq k} d_A(e_A,f(b^jt_i)) \\
		    & \leq & \modulus{g} + \modulus{f} \leq n
\end{eqnarray*}
The number of elements $b^kt_i$ in the support of $h$ can be counted in the following way. Firstly, the number of $i\in I$ for which $\langle b \rangle t_i \cap \Supp{h} \neq \emptyset$ is equal to $s$, which we showed above to be bounded by $n$. Secondly, for each such $i$, recall that there exists $m_1 \leq m_2$ such that $b^jt_i \in \Supp{h}$ implies $m_1 \leq j \leq m_2$. Hence for each $i$ the number of $k \in \Z$ for which $b^kt_i \in \Supp{h}$ is bounded above by $m_2-m_1 \leq \delta_{\langle b \rangle}^B(P)$. So in conclusion we have
\begin{eqnarray*}
d_\Gamma(1,\gamma) & = & K(\Supp{h},z) + \modulus{h} \\
		   & \leq & (n+1)P(\delta_{\langle b \rangle}^B(P)+1) + n^2 \delta_{\langle b \rangle}^B(P) \\
		   & \leq & (n+1)P(2\delta_{\langle b \rangle}^B(P)+1)
\end{eqnarray*}
where $n=d_\Gamma(1,u) + d_\Gamma(1,v)$ and $P=d_B(e_B,z)+n$.
\end{proof}

\begin{thm}\label{thm:wreath conj length}
Suppose $A$ and $B$ are finitely generated. Let $u=(f,b),v=(g,c) \in \Gamma=A\wr B$, with $b$ and $c$ of infinite order, and set $n = d_\Gamma(1,u)+d_\Gamma(1,v)$. Then $u,v$ are conjugate if and only if there exists a conjugator $\gamma \in \Gamma$ such that 
				$$d_\Gamma (1,\gamma) \leq (n+1)P(2\delta_{\langle b \rangle}^B(P)+1)$$
where $P=2n$ if $(f,b)$ is not conjugate to $(1,b)$ and $P=n+\CLF_B(n)$ otherwise.
\end{thm}

\begin{proof}
By Lemma \ref{lem:wreath conjugator with short B part} we can find a conjugator  $\gamma = (h,z)$ which satisfies the inequality $d_B(e_B,z) \leq \CLF_B(n)$ if $(f,b)$ is conjugate to $(1,b)$ or $d_B(e_B,z) \leq n$ otherwise. Therefore if we set $P=n+\CLF_B(n)$ if $(f,b)$ is conjugate to $(1,b)$ and $P=2n$ otherwise then the result follows immediately by applying Lemma \ref{lem:wreath conj bound depending on B part} to the conjugator $\gamma$ obtained from Lemma \ref{lem:wreath conjugator with short B part}.
\end{proof}

When we look at elements whose $B$--components may be of finite order, we can still obtain some information on the conjugator length. Theorem \ref{thm:wreath conj length} does the work when we look at elements in $A \wr B$ such that the $B$--components are of infinite order. However, if they have finite order we need to understand the size of the conjugators $\alpha_i$ as in Proposition \ref{prop:Matthews-conjugacy in wreath products} and Lemma \ref{lem:coset reps}. When the order of $b$ is finite, the construction of the function $h$ by Matthews in Proposition \ref{prop:Matthews-conjugacy in wreath products} will work for any conjugator $\alpha_{t_i}$ between $\pi_{t_i}^{(z)}(g)$ and $\pi_{t_i}(f)$. Then, since 
				$$|\pi_{t_i}^{(z)}(g)| + \modulus{\pi_{t_i}(f)} \leq \modulus{g} + \modulus{f} \leq n$$
where $n=d_\Gamma(1,u) + d_\Gamma(1,v)$, for each coset representative $t_i$ and each $b^k \in \langle b \rangle$ we have
				$$d_B(e_B,h(b^kt_i)) \leq \modulus{f} + \modulus{g} + \textrm{CLF}_A(n)\leq n + \textrm{CLF}_A(n)\textrm{.}$$
With the aid of the conjugacy length function for $A$ we can therefore give the following:

\begin{lemma}\label{lem:wreath torsion bound depending on B part}
Suppose $u=(f,b),v=(g,c)$ are conjugate elements in $\Gamma=A \wr B$ and let $n = d_\Gamma(1,u)+d_\Gamma(1,v)$. Suppose also that $b$ and $c$ are of finite order $N$. If $\gamma = (h,z)$ is a conjugator for $u$ and $v$ in $\Gamma$ then
				$$d_\Gamma(1,\gamma) \leq  P(N+1)(2n+\CLF_A(n)+1)$$
where $P = d_B(e_B,z)+n$.
\end{lemma}

\begin{proof}
For the most part this proof is the same as for Lemma \ref{lem:wreath conj bound depending on B part}. It will differ in two places. As mentioned above, we obtain 
				$$d_B(e_B,h(b^kt_i)) \leq \modulus{f} + \modulus{g} + \textrm{CLF}_A(n)$$
for each coset representative $t_i$ and $b^k \in \langle b \rangle$. By a similar process as that in Lemma \ref{lem:wreath conj bound depending on B part} we deduce the upper bound
				$$\modulus{h} \leq nN(n+\CLF_A(n))\textrm{.}$$

The second place where we need to modify the proof is in the calculation of an upper bound for the length of each path $p_i$. Since $b$ is of finite order, each coset will give a loop in $\Cay(B,X)$. We will let $p_i$ run around this loop, so its length will be bounded above by $Nd_B(t_i, bt_i)$. As before we get $d_B(t_i,bt_i) \leq n$, so in the upper bound obtained for $K(\Supp{h},z)$ we need only replace the distortion function $\delta_{\langle b \rangle}^B$ by the order $N$ of $b$ in $B$. Thus
				$$K(\Supp{h},z) \leq P(N+1)(n+1)$$
where $P=d_B(e_B,z)+n$. Combining this with the upper bound above for $\modulus{h}$ we get
\begin{eqnarray*}
d_\Gamma(1,\gamma) &\leq& P(N+1)(n+1) + nN(n+\CLF_A(n)) \\
	&\leq &  P(N+1)(2n+\CLF_A(n)+1)
\end{eqnarray*}
proving the Lemma.
\end{proof}

We finish this section by applying Lemma \ref{lem:wreath conjugator with short B part} and Lemma \ref{lem:wreath torsion bound depending on B part} to give the complete picture for the length of short conjugators in the wreath product $A \wr B$.

\begin{thm}\label{thm:wreath torsion conj length}
Suppose $A$ and $B$ are finitely generated groups. Let $u=(f,b),v=(g,c) \in \Gamma$ where the order of $b$ and $c$ is $N \in \N \cup \{\infty\}$. Then $u,v$ are conjugate if and only if there exists a conjugator $\gamma \in \Gamma$ such that either
\begin{equation*}
\begin{array}{ll}d_\Gamma (1,\gamma) \leq  P(N+1)(2n+\CLF_A(n)+1) & \textrm{if $N$ is finite; or}\\
d_\Gamma (1,\gamma) \leq (n+1)P(2\delta_{\langle b \rangle}^B(P)+1) & \textrm{if $N=\infty$,}\end{array}
\end{equation*}
where  $n = d_\Gamma(1,u)+d_\Gamma(1,v)$ and $P=2n$ if $(f,b)$ is conjugate to $(1,b)$ or $P=n+\CLF_B(n)$ otherwise.
\end{thm}

\subsection{Lower Bounds for Lengths of Short Conjugators}\label{subsec:lower bounds}

We saw in Section \ref{sec:short conj in wreath - upper bound} that the distortion of cyclic subgroups plays an important role in the upper bound we determined for the conjugacy length function. We will make use of the distortion to determine a lower bound as well. Firstly, however, we give a straightforward lower bound. In the following, let $\modulus{.}_B$ denote word length in the finitely generated group $B$ and $\modulus{.}$ without the subscript denote word length in $A \wr B$.

\begin{prop}\label{prop:CLF_B lower bound for wreaths}
Let $A$ and $B$ be finitely generated groups. Then
		$$\CLF_{A \wr B}(n) \geq \CLF_B(n)\textrm{.}$$
\end{prop}

\begin{proof}
Let $n \in \N$. The value $\CLF_B(n)$ is defined to be the smallest integer such that whenever $b,c$ are conjugate elements in $B$ and satisfy $\modulus{b}_B+\modulus{c}_B \leq n$ then there is a conjugator $z \in B$ such that $\modulus{z}_B \leq \CLF_B(n)$. Let $b_n,c_n$ be elements which realise this minimum. That is:
\begin{enumerate}
\item $\modulus{b_n}_B + \modulus{c_n}_B \leq n$; and
\item a minimal length conjugator $z_n \in B$ satisfies $\modulus{z_n}_B = \CLF_B(n)$.
\end{enumerate}
Consider the elements $u_n=(1,b_n)$ and $v_n=(1,c_n)$ in $A \wr B$, where $1$ represents the trivial function. Then by Lemma \ref{lem:wreath metric}
		$$\modulus{u_n}+\modulus{v_n} \leq n\textrm{.}$$
Any conjugator $(h,x)$ must satisfy $h^{b_n}=h$ and $b_nx=xc_n$. We may take $h=1$ since any non-trivial function $h$ (the existence of such a conjugator is only possible when $b_n$ is of finite order) will lead to a larger conjugator. Thus a minimal length conjugator for $u_n$ and $v_n$ will have the form $(1,x)$ where $x$ can be chosen to be any conjugator for $b_n$ and $c_n$. In particular, this shows that the minimal length conjugator for $u_n$ and $v_n$ has length $\CLF_B(n)$.
\end{proof}

\subsubsection{Distorted elements}

Let $B$ be a finitely generated group containing an element $x$ of infinite order. If the centraliser of $x$ in $B$, denoted $Z_B(x)$, is sufficiently large relative to $\langle x \rangle$ (see Lemma \ref{lem:clf wreath lower bound}), then we can use the distortion of $\langle x \rangle$ in $B$ to construct two sequences of functions from $B$ to $A$ that allow us to demonstrate a lower bound on the conjugacy length function of $A \wr B$ in terms of this distortion. Given any element $b$ of infinite order in $B$, by taking $x=b^3$ we ensure that $x$ has sufficiently large centraliser in order to apply Lemma \ref{lem:clf wreath lower bound}. Since the distortion of $\langle b \rangle$ in $B$ is (roughly) a third of the distortion of $\langle x \rangle$, we can conclude that the distortion function of any cyclic subgroup in $B$ provides a lower bound for the conjugacy length function of $A \wr B$.

\begin{thm}\label{thm:clf wreath lower bound}
Let $A$ and $B$ be finitely generated groups and let $b \in B$ be any element of infinite order. Then
		$$\CLF_{A \wr B}\big(4n + 4 + 10\modulus{b}_B\big) \geq \frac{4}{3}\delta_{\langle b \rangle}^B\left(n\right)-4.$$
\end{thm}

In order to prove Theorem \ref{thm:clf wreath lower bound} we will use the following closely-related Lemma:

\begin{lemma}\label{lem:clf wreath lower bound}
If the set $\{y \in Z_B(x) : y^2 \notin \langle x \rangle \}$ is non-empty then
		$$\CLF_{A \wr B}\big(4(n + L_x + 1) + 2\modulus{x}_B\big) \geq 4\delta_{\langle x \rangle}^B\left(n\right)$$
where $L_x=\min \{ \modulus{y}_B : y \in Z_B(x), \ y^2 \notin \langle x \rangle \}$.
\end{lemma}

\begin{proof}[Proof of Theorem \ref{thm:clf wreath lower bound}]
To obtain the Theorem we need to apply Lemma \ref{lem:clf wreath lower bound}, taking $x=b^3$. Then $b \in Z_B(x)$ and $b^2 \notin \langle x \rangle$, hence $L_x \leq \modulus{b}$. The distortion function for $b$ satisfies $\delta_{\langle b \rangle}^B(n) \geq \frac{1}{3}\delta_{\langle x \rangle}^B(n)-1$. Hence the Theorem follows by application of Lemma \ref{lem:clf wreath lower bound}.
\end{proof}

\begin{proof}[Proof of Lemma \ref{lem:clf wreath lower bound}]
Take an element $y$ in $Z_B(x)$ which realises this minimum. Let $a$ be any element in the chosen generating set of $A$ and consider two functions $f_n,g_n:B\rightarrow A$ which take values of either $e_A$ or $a$ and which have the following supports:
$$\Supp{f_n}=\{e_B , y \}$$
$$\Supp{g_n}=\{x^{-\delta(n)}, x^{\delta(n)}y\}$$
where $\delta(n)=\delta_{\langle x \rangle}^B(n)$.
We use these functions to define a pair of conjugate elements: $u_n=(f_n,x),v_n=(g_n,x)$. First we will show that the sum of the sizes of these elements grows with $n$. Observe that
$$\modulus{u_n}+\modulus{v_n} = K(\Supp{f_n},x)+2+K(\Supp{g_n},x)+2$$
where the notation is as in Lemma \ref{lem:wreath metric}. For every $n$, 
\begin{align*}
& K(\Supp{f_n},x)  \leq  2\modulus{y}_B + \modulus{x}_B \\
n \leq & K(\Supp{g_n},x)  \leq 2\modulus{y}_B + 4n +\modulus{x}_B \textrm{.}
\end{align*}
Hence $n \leq \modulus{u_n}+\modulus{v_n} \leq 4n + 4L_x + 2\modulus{x}_B + 4$. As an example of a conjugator we may take $\gamma_n=(h_n,e_B)$, where $h_n$ is given by
\begin{align*}
h_n(x^iy)=a \ &\ \textrm{ if } \ 0\leq i \leq \delta(n)-1;\\
h_n(x^{-i})=a^{-1} \ &\ \textrm{ if } \ 1 \leq i \leq \delta(n);\\
h_n(b)=e_A \ &\ \textrm{ otherwise.}
\end{align*}
We will now verify that this is indeed a conjugator. To do so, we need to verify that $f_nh_n^x = h_ng_n$. We need only check this holds for elements of the form $x^i y$ or $x^{-i}$ for $0\leq i \leq \delta(n)$ since otherwise both sides evaluate to the identity. The reader can verify that
		$$f_n(x^iy)h_n(x^{i-1}y)=a=h_n(x^iy)g_n(x^iy)$$
whenever $0\leq i \leq \delta(n)$. Provided $1\leq i \leq \delta(n)-1$ we get
		$$f_n(x^{-i})h_n(x^{-i-1})=a^{-1}=h_n(x^{-i})g_n(x^{-i})$$
and in the last two cases, that is for $i \in \{0,\delta(n)\}$, both sides equal the identity.

\begin{figure}\vspace{2mm}
\labellist \small \hair 4pt
		\pinlabel $\langle x\rangle$ [l] at 447 78
\tiny		\pinlabel $e_B$ [b] at 245 147
\hair 8pt		\pinlabel $y$ [t] at 234 88
		\pinlabel $x^{-1}$ [t] at 235 145
		\pinlabel $x^{-\delta(n)}$ [t] at 31 73
		\pinlabel $x^{\delta(n)-1}y$ [l] at 392 16
		\pinlabel $x^{\delta(n)}y$ [l] at 402 3
	\endlabellist
\centering\includegraphics[width=9cm]{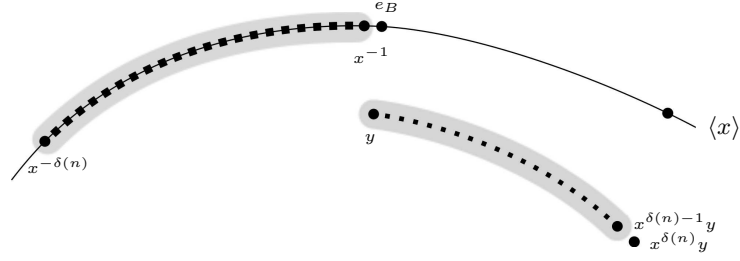}
\caption[The support of $h_n$]{The support of $h_n$ is the union of the two shaded regions.}
\end{figure}

Now we will show that any conjugator for $u_n$ and $v_n$ will have to have size bounded below by $4\delta(n)$. This is done by showing the support of the associated function will contain at least $2\delta(n)$ elements. We will first give the family of elements $z \in B$ for which there exists a conjugator for $u_n$ and $v_n$ of the form $(h,z)$ for some function $h$. Suppose $z$ is some such element. Lemma \ref{lem:coset reps} then tells us what the corresponding function $h$ will look like. In particular, in order for the support of $h$ to be finite, we must have that 
	$$\Supp{f_n} \cap \langle x \rangle t \neq \emptyset\ \ \textrm{ if and only if }\ \ z\Supp{g_n} \cap \langle x \rangle t \neq \emptyset$$
for any $t \in B$ (note that this does not apply in general, but it does here because the functions have been designed so their supports intersect each coset with at most one element). Hence $z\Supp{g_n}$ will intersect the cosets $\langle x \rangle y^i$ exactly once for each $i = 0,1$ and it will not intersect any other coset. Let $\sigma $ be the permutation of $\{0,1\}$ such that $zx^{\delta(n)}y \in \langle x \rangle y^{\sigma(1)}$ and $zx^{-\delta(n)} \in \langle x \rangle y^{\sigma(0)}$. Since $y$ is in the centraliser of $x$ it follows that $z \in \langle x \rangle y^{\sigma(i)-i}$ for each $i$. If $\sigma(i)\neq i$ then this implies that $z \in \langle x \rangle y^{-1} \cap \langle x \rangle y$, so $y^2 \in \langle x \rangle$, contradicting our choice of $y$. Hence $\sigma(i)=i$ for $i \in \{0,1\}$, and  thus $z \in \langle x \rangle$.

The support of $g_n$ was chosen in such a way that it is sufficiently spread out in the two cosets of $\langle x \rangle$. It means that shifting $\Supp{g_n}$ by any power of $x$ does not prevent the support of $h$ needing at least $2\delta(n)$ elements. In particular, if $z=x^k$ for some $-\delta(n)<k< \delta(n)$ then the support of $h$ will consist of elements $x^iy$ for $0\leq i <\delta(n)+k$ and $x^{-i}$ for $1\leq i \leq \delta(n)-k$. Here the support has precisely $2\delta(n)$ elements. If $k$ lies outside this range then the support will contain at least as many as $2\delta(n)$ elements, for example if $k\geq \delta(n)$ then the support will consist of elements $x^iy$ for $0\leq i <\delta(n)+k$ as well as $x^i$ for any $i$ satisfying $0\leq i <-\delta(n)+k$. This implies that, by Lemma \ref{lem:wreath metric}, any conjugator for $u_n$ and $v_n$ will have to have size at least $4\delta(n)$, providing the required lower bound for the conjugacy length function.
\end{proof}

In \cite{SaleThesis} the author shows that groups of the form $A\wr \BS{q}$, where $\BS{q}$ is a Baumslag-Solitar group, have an exponential conjugacy length function. The lower bound is obtained by using the methods in the proof of Theorem \ref{thm:clf wreath lower bound}.

Osin \cite{Osin01} has described the distortion functions of subgroups of finitely generated nilpotent groups. In particular, for a $c$--step nilpotent group $N$ his result implies that the maximal distortion of a cyclic subgroup of $N$ will be $n^c$, and this occurs when the subgroup is contained in the centre of $N$. A consequence of Theorem \ref{thm:wreath conj length}, Theorem \ref{thm:clf wreath lower bound} and Osin's work is that when restricting to elements in $A \wr N$ not conjugate to an element of the form $(1,b)$, the (restricted) conjugacy length function will be $n^\alpha$, where $\alpha \in [c,c+2]$. However, since we do not yet know the conjugacy length function of a general $c$--step nilpotent group, apart from this lower bound we cannot estimate the conjugacy length function of $A \wr N$. In the particular case when $N$ is $2$--step nilpotent we know its conjugacy length function is quadratic by Ji, Ogle and Ramsey \cite{JOR10}. Hence the conjugacy length function of $A\wr N$, when $N$ is torsion-free, is $n^\alpha$ for some $\alpha \in [2,7]$.

\subsubsection{Using the area of triangles}

We will show that the conjugacy length function of wreath products $A \wr B$, where $A$ is finitely generated and $B$ contains a copy of $\Z^2$, are non-linear, and in particular are at least quadratic. Combined with Theorem \ref{thm:wreath conj length} we learn, for example, that $\CLF_{A \wr \Z^r}(n) \asymp n^\alpha$ for some $\alpha \in [2,3]$ whenever $r \geq 2$. The methods used here differ to those used above in that we do not use subgroup distortion. Instead we rely on the area of triangles in $\Z^2$ being quadratic with respect to the perimeter length.

\begin{figure}[b!]
\labellist \small \hair 5pt
	\pinlabel $y^k\Supp{g_n}$ [r] at 0 0
	\pinlabel $\Supp{f_n}$ [r] at 0 444
	\pinlabel $n$ [b] at 352 460
	\pinlabel $n$ [r] at 240 332
\endlabellist
\centering\includegraphics[width=6cm]{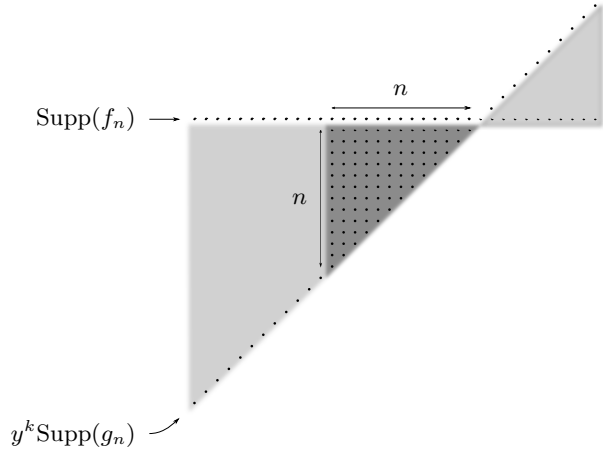}
\caption{The shaded region indicates $\Supp{h}$, while the dark shaded region is a triangle contained in $\Supp{h}$ which contains $\frac{1}{2}n(n+1)$ elements.}
\end{figure}

Suppose $x$ and $y$ generate a copy of $\Z^2$ in $B$. For each $n \in \N$ let $f_n,g_n:B \to A$ be functions which take values of either $e_A$ or $a$, where $a$ is an element of a generating set for $A$, and have supports given by
		$$\Supp{f_n}=\{x^{-n},\ldots , x^{-1} , e_B , x , \ldots , x^n\},$$
		$$\Supp{g_n}=\{x^{-n}y^{-n} , \ldots , x^{-1}y^{-1},e_B,xy,\ldots x^ny^n\}.$$
Consider the two elements $u_n=(f_n,y)$ and $v_n=(g_n,y)$. These are conjugate via the element $(h_n,e_B)$, where $h_n$ is defined by
\begin{align*}
h_n(x^iy^j)=a \ & \ \textrm{ if } \ 0<i\leq n \ \textrm{ and } \ 0\leq j <i;\\
h_n(x^iy^j)=a^{-1} \ &\ \textrm{ if } \ -n \leq i < 0 \ \textrm{ and } \ i \leq j < i;\\
h_n(z)=e_A \ &\ \textrm{ otherwise.}
\end{align*}
It is clear that the sizes of $u_n$ and $v_n$ grow linearly with $n$. In particular, using Lemma \ref{lem:wreath metric}, one can verify that
		$$4n+2 \leq \modulus{u_n}\leq 4n\modulus{x}_B+\modulus{y}_B+2n+1$$
		$$4n+2 \leq \modulus{v_n}\leq 4n\modulus{xy}_B+\modulus{y}_B+2n+1.$$
Suppose that $(h,z)$ is any conjugator for $u_n$ and $v_n$. First we claim that $z$ must be a power of $y$. This follows from a similar argument as in the proof of Theorem \ref{thm:clf wreath lower bound}. By Lemma \ref{lem:coset reps}, given $z$ we can construct the function $h$. The support of $h$ will be finite only if for any $t \in B$
\begin{equation}\label{eq:suport of h}
\Supp{f_n} \cap \langle y \rangle t \neq \emptyset\ \ \textrm{ if and only if }\ \ z\Supp{g_n} \cap \langle y \rangle t \neq \emptyset.
\end{equation}
First observe that $z \in \langle x,y \rangle = \Z^2$ since otherwise we would have $\Supp{f_n}\cap \langle y \rangle z =\emptyset$ while on the other hand $z \in z\Supp{g_n}\cap \langle y \rangle z$. By the nature of $\Z^2$ we can see that if $z$ acts on it by translation in any direction except those parallel to $y$ then (\ref{eq:suport of h}) cannot hold, implying $h$ will have infinite support. Thus any conjugator for $u_n$ and $v_n$ must be of the form $(h,y^k)$ for some integer $k$.

A simple geometric argument now gives us a quadratic lower bound on the size of $(h,y^k)$ relative to $n$. The support of $h$ will be contained in the cosets $\langle y \rangle x^i$ for $-n \leq i \leq n$. Within each coset it will include precisely one of $x^i$ or $x^ry^{i+k}$, as well as all elements in between. Specifically, $x^iy^j \in \Supp{h}$ if and only if either $0\leq j < i+k$ or $i+k \leq j < 0$. A triangle of elements is therefore contained in the support of $h$. Regardless of what value of $k$ is chosen this triangle can be chosen to have at least $n$ elements along the base and side, giving a minimum of $\frac{1}{2}n(n+1)$ elements in $\Supp{h}$. Thus any conjugator must have size bounded below by $n(n+1)$.

\begin{thm}\label{thm:Z^r wr Z^r lower bound on clf}
Let $A$ and $B$ be finitely generated groups and suppose that $B$ contains a copy of $\Z^2$ generated by elements $x$ and $y$. The conjugacy length function of $A \wr B$ satisfies
		$$\CLF_{A\wr B}\big(4n(\modulus{x}+\modulus{xy}+1)+2\modulus{y}+2\big)\geq n^2+n.$$
When $B=\Z^r$ for some $r\geq 2$ we get the following upper and lower bounds:
		$$\frac{(n+14)(n-2)}{256}\leq \CLF_{A \wr \Z^r}(n)\leq 7n(n+1)(14n+1).$$
\end{thm}

\begin{proof}
The lower bounds follow from the argument preceding the Proposition. The upper bound for the second expression is an immediate consequence of Theorem \ref{thm:wreath conj length}, using the facts that $\CLF_{\Z^r}$ is identically zero and the distortion function for cyclic subgroups in $\Z^r$ is the identity function.
\end{proof}

\begin{samepage}The use of area in this way to provide a lower bound on the conjugacy length function raises the question of whether, in general, one can use the Dehn function of a group $B$ to provide a lower bound for $\CLF_{A\wr B}$.

\vspace{1mm}
\begin{question}
Let $\mathrm{Area}_B:\N \to \N$ be the Dehn function of a finitely presented group $B$. Is it true that $\CLF_{A\wr B}(n) \succeq \mathrm{Area}_B(n)$ for any finitely generated group $A$?
\end{question}
\end{samepage}

\section{Free solvable groups}\label{sec:free solvable}

We now turn our attention to study the conjugacy length function of free solvable groups. The method involves using the Magnus embedding to see the free solvable group $S_{r,d+1}$ as a subgroup of the wreath product $\Z^r \wr S_{r,d}$. Results of Section \ref{sec:conjugacy in wreath products} will be central to estimating the length of a conjugating element.

\subsection{The Magnus embedding}

We briefly recap two equivalent definitions of the Magnus embedding, one algebraic, one geometric. For more details see \cite{Sale12magnusqi}.

\subsubsection{Definition via Fox calculus}

In order to define the Magnus embedding we need to first define Fox derivatives. These are derivations on a group ring $\Z(F)$, where $F$ is the free group on generators $X=\{x_1,\ldots , x_r\}$. For each generator we can define a unique derivation $\frac{\partial}{\partial x_i}$ which satisfies
				$$\frac{\partial x_j}{\partial x_i} = \delta_{ij}$$
where $\delta_{ij}$ is the Kronecker delta.

Let $N$ be a normal subgroup of $F$ and denote the quotient homomorphism by $\alpha:F \rightarrow F/N$. The Magnus embedding gives a way of recognising $F/N'$, where $N'$ is the derived subgroup of $N$, as a subgroup of the wreath product $M(F/N)=\Z^r \wr F/N$.

Given a derivation $\mathcal{D}$ on $\Z(F)$, denote by $\mathcal{D}^\star:\Z(F) \rightarrow \Z(F/N)$ the composition of $\mathcal{D}$ with $\alpha$ (extended linearly over $\Z(F)$). Consider the group ring $\Z(F/N)$ and let $\mathcal{R}$ be the free $\Z(F/N)$--module with generators $t_1, \ldots , t_r$. We define a homomorphism
				$$\varphi : F \longrightarrow M(F/N)=\left( \begin{array}{cc}F/N & \mathcal{R} \\ 0 & 1 \end{array}\right)=\left\{ \left(\begin{array}{cc} g & a \\ 0 & 1 \end{array}\right) \mid g \in F/N, a \in \mathcal{R} \right\}$$
by
				$$\varphi(w) = \Magnus{w}.$$
Note that $M(F/N) \cong \Z^r \wr F/N$. Magnus \cite{Magn39} recognised that the kernel of $\varphi$ is equal to $N'$ and hence $\varphi$ induces an injective homomorphism from $F/N'$ to $M(F/N)$ which is known as the \emph{Magnus embedding}. In the rest of this paper we will use $\varphi$ to denote both the homomorphism defined above and the Magnus embedding it induces.

\subsubsection{Geometric definition}\label{sec:geometric definition}

Take a word $w$ in $F$ and construct the path $\rho_w$ read out by this word in the Cayley graph $\Cay(F/N,\overline{X})$ of $F/N$, where $\overline{X}$ is the image of the generating set $X$ in the quotient $F/N$. 

Let $E$ be the edge set of the Cayley graph $\Cay(F/N,\overline{X})$. Define a function $\pi_w : E \rightarrow \Z$ such that for each edge $(g,gx) \in E$ the value of $\pi_w(g,gx)$ is equal to the net number of times the path $\rho_w$ traverses this edge --- for each time the path travels from $g$ to $gx$ count $+1$; for each time the path goes from $gx$ to $g$ count $-1$.

Given $w \in F$ we will use $\pi_w$ to define a function $P_w:F/N\to \Z^r$ in the natural way:
		$$P_w(g)=\big(\pi_w(g,gx_1),\ldots , \pi_w(g,gx_r)\big).$$
Define the \emph{geometric Magnus embedding} to be the function $\maggeo:F \to \Z^r \wr F/N$  such that $\maggeo(w)=(P_w,\alpha(w))$ for $w \in F$.

\begin{thm}[\hspace{-0.2mm}\cite{Sale12magnusqi}]\label{thm:geometric magnus} 
The two definitions of the Magnus embedding, $\varphi$ and $\maggeo$, are equivalent.
\end{thm}

\subsubsection{Properties of the Magnus embedding and Fox calculus}

In \cite{Sale12magnusqi} we used Theorem \ref{thm:geometric magnus} to demonstrate that the Magnus embedding is 2-bi-Lipschitz with respect to word metrics determined by canonical generating sets. In particular, let $d_{F/N'}$ denote the word metric in $F/N'$ with respect to the generating set $\alpha(X)$ and let $d_M$ denote the word metric on the wreath product $M(F/N)$ with respect to the generating set  
				$$\left\{ \SolMX{\alpha(x_1)}{0} , \ldots , \SolMX{\alpha(x_r)}{0} , \SolMX{1}{t_1} , \ldots , \SolMX{1}{t_r} \right\}\textrm{.}$$
This generating set coincides with the one on the wreath product described in Section \ref{sec:wreath word length}, with respect to the generating set $\alpha(X)$ on $F/N$ and the standard generating set for $\Z^r$.

\begin{thm}[\hspace{-0.2mm}\cite{Sale12magnusqi}]\label{thm:F/N' undistorted in M(F/N)}
The subgroup $\varphi(F/N')$ is undistorted in $M(F/N)$. To be precise, for each $g \in F/N'$
				$$\frac{1}{2}d_{F/N'}(1,g) \leq d_M(1,\varphi(g))\leq 2d_{F/N'}(1,g)\textrm{.}$$
\end{thm}

We quickly recall the fundamental formula of Fox calculus. Let $\varepsilon:\Z(F) \to \Z$ be the homomorphism sending every element of $F$ to $1$.

\begin{lemma}[Fundamental formula of Fox calculus {\cite[(2.3)]{Fox53}}]\label{lem:fundamental theorem of fox calculus}
Let $a \in \Z(F)$. Then
				$$a-\varepsilon(a)1=\sum_{i=1}^{r} \ddxi{a} (x_i -1)\mathrm{.}$$
\end{lemma}

In Section \ref{sec:Conjugacy in free solvable groups} we will require the following result:

\begin{lemma}[See also Gruenberg {\cite[\S 3.1 Theorem 1]{Grue67}}]\label{lem:kernel of alpha^star}
An element of the kernel of $\bar{\alpha}^\star : \Z(F/N')\rightarrow \Z(F/N)$ can be written in the form
				$$\sum_{j=1}^m r_j(h_j -1)$$
for some integer $m$, where $r_j \in F/N'$ and $h_j \in N/N'$ for each $j = 1 , \ldots , m$.
\end{lemma}

\begin{proof}
Take an arbitrary element $a$ in the kernel of $\bar{\alpha}^\star$. Suppose we can write
				$$a=\sum_{g \in F/N'} \beta_g g$$
where $\beta_g \in \Z$ for each $g \in F/N'$. Fix a coset $xN$. Then
				$$\sum_{\bar{\alpha}(g)=xN}\beta_g = 0$$
since this is the coefficient of $xN$ in $\bar{\alpha}^\star(a)$. Notice that $\bar{\alpha}(g)=xN$ if and only if there is some $h \in N$ such that $g = xh$. Thus the sum can be rewritten as
				$$\sum_{h \in N\setminus\{1\}} \beta_{xh} = -\beta_x\textrm{.}$$
This leads us to
				$$\sum_{\bar{\alpha}(g)=xN} \beta_g g = \sum_{h \in N\setminus\{1\}}\beta_{xh}x(h-1)$$
which implies the Lemma after summing over all left-cosets.
\end{proof}

\subsection{Subgroup Distortion}\label{sec:subgroup distortion}

We saw in Theorem \ref{thm:wreath conj length} that in order to understand the conjugacy length function of a wreath product $A \wr B$ we need to understand the distortion function for infinite cyclic subgroups in $B$.

Recall subgroup distortion is studied up to an equivalence relation of functions. For functions $f,g:\N \rightarrow [0,\infty)$ we write $f \preceq g$ if there exists an integer $C>0$ such that $f(n) \leq Cg(Cn)$ for all $n \in \N$. The two functions are equivalent if both $f \preceq g$ and $g \preceq f$. In this case we write $f \asymp g$. 

We will see that all cyclic subgroups of free solvable groups $S_{r,d}$ are undistorted. This is not always the case in finitely generated solvable groups. For example, in the solvable Baumslag-Solitar groups $BS(1,q)=\langle a,b \mid aba^{-1}=b^q \rangle$ the subgroup generated by $b$ is at least exponentially distorted since $b^{q^n}=a^nba^{-n}$. Because this type of construction doesn't work in $\Z \wr \Z$ or free metabelian groups it leads to a question of whether all subgroups of these groups are undistorted (see \cite[\S 2.1]{DO11}). Davis and Olshanskii answered this question in the negative, giving, for any positive integer $t$, $2$--generated subgroups of these groups with distortion function bounded below by a polynomial of degree $t$.

The following Lemma is given in \cite[Lemma 2.3]{DO11}.

\begin{lemma}
Let $A,B$ be finitely generated abelian groups. Then every finitely generated abelian subgroup of $A \wr B$ is undistorted.
\end{lemma}

Davis and Olshanskii prove this by showing that such subgroups are retracts of a finite index subgroup of $A \wr B$. A similar process can be applied to finitely generated abelian subgroups of free solvable groups to show that they are undistorted \cite{OlshanskiiPrivate}. Below we give an alternative proof for cyclic subgroups which provides an effective estimate for the constant\footnote{We thank Olshanskii for improving the constant from $2^d$ to $2$.} and which uses only results given in this paper.

\begin{prop}\label{prop:cyclic subgroup distortion in free solvable groups}
Every cyclic subgroup of a free solvable group is undistorted. In particular, suppose $d \geq 1$ and let $x$ be a non-trivial element of $S_{r,d}$. Then
				$$\delta_{\langle x \rangle}^{S_{r,d}}(n) \leq 2n\textrm{.}$$
\end{prop}

\begin{proof}
Let $w$ be a non-trivial element of the free group $F$. There exists an integer $c$ such that $w \in F^{(c)} \setminus F^{(c+1)}$, where we include the case $F^{(0)}=F$. First we suppose that $d=c+1$.

If $c=0$ then we have $x \in F/F'=\Z^r$ and we apply linear distortion in $\Z^r$. If $c>0$ then we take a Magnus embedding $\varphi : S_{r,d} \hookrightarrow \Z^r \wr S_{r,c}$ and observe that since $w \in F^{(c)}\setminus F^{(d)}$ the image of $x$ in $\varphi$ is $(f,1)$ for some non-trivial function $f:S_{r,c}\rightarrow \Z^r$. If $f^k$ denotes the function such that $f^k(b)=kf(b)$ for $b \in S_{r,c}$, then for any $k \in \Z$, since the Magnus embedding is $2$-bi-Lipschitz (Theorem \ref{thm:F/N' undistorted in M(F/N)}),
				$$ d_{S_{r,d}}(1,x^k)  \geq  \frac{1}{2}d_M(1,(f,1)^k) = \frac{1}{2} d_M(1,(f^k,1))\textrm{.}$$
We can apply Lemma \ref{lem:wreath metric} to get
				$$\frac{1}{2} d_M(1,(f^k,1))=\frac{1}{2}\left( K(\Supp{f},1) + \sum_{b \in S_{r,c}} \modulus{kf(b)} \right)$$
and since the image of $f$ lies in $\Z^r$ and $f$ is non-trivial
				$$\sum_{b \in S_{r,c}}\modulus{kf(b)} =\modulus{k}\sum_{b \in S_{r,c}} \modulus{f(b)}\geq \modulus{k}\textrm{.}$$
Hence
				$$ d_{S_{r,d}}(1,x^k)  \geq \frac{1}{2} d_M(1,(f^k,1)) \geq  \frac{1}{2}\modulus{k}\textrm{.}$$
This implies $\delta_{\langle x \rangle}^{S_{r,d}}(n) \leq 2n$.

Now suppose that $d>c+1$. Then we define a homomorphism 
			$$\psi:S_{r,d}\rightarrow S_{r,c+1}$$
by sending the free generators of $S_{r,d}$ to the corresponding free generator of $S_{r,c+1}$. Then, as before, set $x=wF^{(d)}$ and define $y$ to be the image of $x$ under $\psi$. By the construction of $\psi$ we have that $y=wF^{(c+1)}$. Note that $y$ is non-trivial and $\psi$ does not increase the word length. Hence, using the result above for $S_{r,c+1}$, we observe that
				$$d_{S_{r,d}}(1,x^k) \geq d_{S_{r,c+1}}(1,y^k)\geq\frac{1}{2}\modulus{k}\textrm{.}$$
This suffices to show that the distortion function is always bounded above by $2n$.
\end{proof}

\subsection{Conjugacy in Free Solvable Groups}\label{sec:Conjugacy in free solvable groups}

The fact that the conjugacy problem is solvable in free solvable groups was shown by Kargapolov and Remeslennikov \cite{KR66}. The following Theorem was given by Remeslennikov and Sokolov \cite{RS70}. They use it, alongside Matthews' result for wreath products, to show the decidability of the conjugacy problem in $S_{r,d}$.

\begin{thm}[Remeslennikov--Sokolov \cite{RS70}]\label{thm:remeslennikov and sokolov}
Suppose $F/N$ is torsion-free and let $u,v \in F/N'$. Then $\varphi(u)$ is conjugate to $\varphi(v)$ in $M(F/N)$ if and only if $u$ is conjugate to $v$ in $F/N'$.
\end{thm}

The proof of Theorem \ref{thm:Magnus conjugators} more-or-less follows Remeslennikov and Sokolov's proof of the preceding theorem. With it one can better understand the nature of conjugators and how they relate to the Magnus embedding. In particular it tells us that once we have found a conjugator in the wreath product $\Z^r \wr S_{r,d-1}$ for the image of two elements in $S_{r,d}$, we need only modify the function component of the element to make it lie in the image of the Magnus embedding.

\begin{thm}\label{thm:Magnus conjugators}
Let $u,v$ be two elements in $F/N'$ such that $\varphi(u)$ is conjugate to $\varphi(v)$ in $M(F/N)$. Let $g \in M(F/N)$ be identified with $(f,\gamma) \in \Z^r \wr F/N$. Suppose that $\varphi(u)g = g\varphi(v)$. Then there exists $w \in F/N'$ such that $\varphi(w)=(f_w,\gamma)$ is a conjugator.
\end{thm}

\begin{proof}
Let $g \in M(F/N)$ be such that $\varphi(u)g=g\varphi(v)$. Suppose
				$$g = \SolMX{\gamma}{a}=(f,\gamma)$$
for $\gamma \in F/N$ and $a \in \mathcal{R}$. Denote by $\bar{\alpha}:F/N' \rightarrow F/N$ the quotient homomorphism. By direct calculation we obtain the two equations
\begin{eqnarray}
\label{eq:prop:Magnus conjugators:conjugacy equation diagonal}\bar{\alpha}(u)\gamma &=& \gamma \bar{\alpha}(v) \\
\label{eq:prop:Magnus conjugators:conjugacy equation unipotent}\sum_{i=1}^r \dstardxi{u}t_i + \bar{\alpha}(u)a &=& \gamma \sum_{i=1}^r \dstardxi{v}t_i + a
\end{eqnarray}
We now split the proof into two cases, depending on whether or not $\bar{\alpha}(u)$ is the identity element.

\vspace{2mm}
\noindent \underline{\textsc{Case} 1:} $\bar{\alpha}(u)$ is trivial.
\vspace{2mm}

In this case equation (\ref{eq:prop:Magnus conjugators:conjugacy equation unipotent}) reduces to
				$$\sum_{i=1}^r \dstardxi{u}t_i = \gamma \sum_{i=1}^r \dstardxi{v}t_i$$
and it follows that $\varphi(\gamma_0)=(h,\gamma)$ will be a conjugator for $\varphi(u)$ and $\varphi(v)$, where $\gamma_0$ is any lift of $\gamma$ in $F/N'$.

\vspace{2mm}
\noindent \underline{\textsc{Case} 2:} $\bar{\alpha}(u)$ is non-trivial.
\vspace{2mm}

Note that in this case we actually show a stronger result, that any conjugator for $\varphi(u)$ and $\varphi(v)$ must lie in the subgroup $\varphi(F/N')$. This is clearly not necessarily true in the first case.

First conjugate $\varphi(u)$ by $\varphi(\gamma_0)$, where $\gamma_0$ is a lift of $\gamma$ in $F/N'$. This gives us two elements which are conjugate by a unipotent matrix in $M(F/N)$, in particular there exist $b_1, \ldots , b_r$ in $\Z(F/N)$ such that the conjugator is of the form
				$$\varphi(\gamma_0)^{-1}g = \gamma'=\SolMX{1}{b_1t_1 + \ldots + b_rt_r}\textrm{.}$$
Hence the aim now is to show that there is some element $y$ in $N$ such that $\gamma'=\varphi(y)$, in particular $\dstardxi{y}=b_i$ for each $i$. Therefore, without loss of generality, we assume that $\varphi(u)$ and $\varphi(v)$ are conjugate by such a unipotent matrix. 

Assume that $\varphi(u)\gamma'=\gamma'\varphi(v)$. Then equation (\ref{eq:prop:Magnus conjugators:conjugacy equation diagonal}) tells us that $\bar{\alpha}(u)=\bar{\alpha}(v)$. Hence $uv^{-1}=z\in N$. Observe that $\dstardxi{z}=\dstardxi{u}-\dstardxi{v}$, hence from equation (\ref{eq:prop:Magnus conjugators:conjugacy equation unipotent}) we get
				$$(1-\bar{\alpha}(u))b_i = \dstardxi{z}$$
for each $i = 1 , \ldots , r$.

Let $c$ be an element of $\Z(F/N')$ such that $\dstardxi{c}=b_i$ for each $i$. We therefore have the following:
				$$(1-\bar{\alpha}(u))\sum_{i=1}^r \dstardxi{c} (\alpha(x_i)-1)=\sum_{i=1}^r \dstardxi{z}(\alpha(x_i)-1)\textrm{.}$$
We can choose $c$ so that $\varepsilon(c)=0$, and then apply the fundamental formula of Fox calculus, Lemma \ref{lem:fundamental theorem of fox calculus}, to both sides of this equation to get 
				$$(1-\bar{\alpha}(u))c=z-1$$
since $z \in N$ implies $\varepsilon(z)=1$. In $\Z(F/N)$ the right-hand side is $0$. Furthermore, since $F/N$ is torsion-free, $(1-\bar{\alpha}(u))$ is not a zero divisor, so $c$ lies in the kernel of the homomorphism $\bar{\alpha}^\star : \Z(F/N') \rightarrow \Z(F/N)$. Hence there is an expression of $c$ in the following way (see Lemma \ref{lem:kernel of alpha^star}):
				$$c=\sum_{j=1}^m r_j(h_j-1)$$
where $m$ is a positive integer and  for each \mbox{$j = 1, \ldots , m$} we have $r_j\in F/N'$ and $h_j \in N/N'$. Differentiating this expression therefore gives
\begin{eqnarray*}
b_i = \dstardxi{c} &=& \sum_{j=1}^m \left( \dstardxi{r_j} \varepsilon(h_j -1) + \bar{\alpha}(r_j) \dstardxi{(h_j-1)}\right)\\
				&=& \sum_{j=1}^m \bar{\alpha}(r_j)\dstardxi{h_j}.
\end{eqnarray*}
We set 
					$$y = \prod\limits_{j=1}^m r_j h_j r_j^{-1} \in N\textrm{.}$$
Since $h_j \in N/N'$ for each $j$, we have the following equations:
				$$\dstardxi{(h_1 h_2)}=\dstardxi{h_1}+\dstardxi{h_2} , \ \ \ \dstardxi{(r_j h_j r_j^{-1})} = r_j \dstardxi{h_j}\textrm{.}$$
Using these, the condition $\dstardxi{y}=b_i$ can be verified. Hence, $g=\varphi(\gamma_0)\varphi(y)$, so taking $w=\gamma_0 y$ gives us a conjugator $\varphi(\gamma_0 y) = (f_w,\gamma)$ of the required form.
\end{proof}

Theorem \ref{thm:Magnus conjugators} tells us that when considering two conjugate elements in $F/N'$ we may use the wreath product result, Theorem \ref{thm:wreath torsion conj length}, and the fact that the Magnus embedding does not distort word lengths, Theorem \ref{thm:F/N' undistorted in M(F/N)}, to obtain a control on the length of conjugators in $F/N'$ in terms of the conjugacy length function of $F/N$. Recall that in $A\wr B$ the conjugacy length function of $B$ plays a role only when we consider elements conjugate to something of the form $(1,b)$. Hence, with the use of the following Lemma, we can obtain an upper bound on the conjugacy length function of $F/N'$ which is independent of the conjugacy length function of $F/N$ and depends only on the distortion of cyclic subgroups in $F/N$.

\begin{lemma}\label{lem:image not (1,b)}
Let $(f,b)$ be a non-trivial element in the image $\varphi(F/N')$ where $b$ is of infinite order in $F/N$. Then $(f,b)$ is not conjugate to any element of the form $(1,c)$ in $\Z^r \wr F/N$.
\end{lemma}

\begin{proof}
First note that we may assume $b\neq e$, since otherwise for $(f,e)$ to be conjugate to $(1,c)$, $c$ would have to be $e$ and $f=1$. Suppose, for contradiction, that $(1,c)=(h,z)^{-1}(f,b)(h,z)$. Then
\begin{equation}\label{eq:image not (1,b)}
0 = -h(zg) + f(zg) + h(b^{-1}zg)
\end{equation}
for every $g \in F/N$. We will show that for this to be true, the support of $h$ must be infinite. Write $f$ and $h$ in component form, that is 
		$$f(g)=\big(f_1(g),\ldots , f_r(g)\big)\ \textrm{ and } \ h(g)=\big(h_1(g),\ldots , h_r(g)\big)$$
for $g \in F/N$ and where $f_i,h_i:F/N \to \Z$. Recall that the generators of $F$ are $X=\{x_1,\ldots ,x_r\}$. We will abuse notation, letting this set denote a generating set for $F/N$ as well. Consider the function $\Sigma_f:F/N \to \Z$ defined by
		$$\Sigma_f(g)=\sum_{i=1}^r f_i(g) - \sum_{i=1}^r f_i(gx_i^{-1})$$
for $g \in F/N$ and similarly $\Sigma_h$ for $h$ in place of $f$. Since $(f,b)$ is in the image of the Magnus embedding, using the geometric definition of Section \ref{sec:geometric definition}, it gives a path $\rho$ from the identity to $b$ in the Cayley graph of $F/N$. The geometric definition also tells us that the function $f_i$ counts $+1$ when $\rho$ travels from $g$ to $gx_i$ and $-1$ each time it goes from $gx_i$ to $g$. Therefore $\Sigma_f(g)$ counts the net number of times this path leaves the vertex labelled $g$ and in particular we deduce that
		$$\Sigma_f(e)=1, \ \ \Sigma_f(b)=-1 \ \textrm{ and } \ \Sigma_f(g)=0 \ \textrm{ for } \ g \neq e,b.$$
From equation (\ref{eq:image not (1,b)}) we get $\Sigma_f(g)=\Sigma_h(g)-\Sigma_h(b^{-1}g)$ for every $g \in F/N$. Note that if the support of $h$ is to be finite, the support of $\Sigma_h$ must also be finite.

Consider any coset $\langle b \rangle t$ in $F/N$. Suppose first that $t$ is not a power of $b$. Then $\Sigma_f$ is identically zero on $\langle b \rangle t$. Since $0=\Sigma_f(b^kt)=\Sigma_h(b^kt)-\Sigma_h(b^{k-1}t)$ implies that $\Sigma_h$ is constant on $\langle b \rangle t$, if $h$ is to have finite support then this must always be zero since $b$ is of infinite order. Looking now at $\langle b \rangle$, we similarly get $\Sigma_h(b^k)=\Sigma_h(b^{k-1})$ for all $k \neq 0,1$. Again, the finiteness of the support of $h$ implies that $\Sigma_h(b^k)=0$ for all $k\neq 0$. However, $1=\Sigma_f(e)=\Sigma_h(e)-\Sigma_h(b^{-1})$, so $\Sigma_h(e)=1$. To summarise:
		$$\Sigma_h(g)\neq 0 \ \textrm{ if and only if } \ g=e.$$
		
We can use this to construct an infinite path in $\Cay(F/N,X)$, which contains infinitely many points in $\Supp{h}$ and thus obtain a contradiction. Start by setting $g_0$ to be $e$. Since $\Sigma_h(e)=1$ there is some $x_{i_1} \in X$ such that either $h_{i_1}(e)\neq 0$ or $h_{i_1}(x_{i_1}^{-1})\neq 0$. If the former is true then let $g_1=x_{i_1}$; in the latter case take $g_1=x_{i_1}^{-1}$. Since $\Sigma_h(g_1)=0$ there must be some adjacent vertex $g_2$ such that either $g_2=g_1x_{i_2}$ and $h_{i_2}(g_1)\neq 0$ or $g_2=g_1x_{i_2}^{-1}$ and $h_{i_2}(g_2)\neq 0$. We can extend this construction endlessly, building an infinite sequence $(g_m)$ for $m \in \Z\cup \{0\}$. Furthermore, for each edge in the induced path in $\Cay(F/N,X)$, at least one of its end points will be in the support of $h$. Since $h_i(g_m)$ is finite for every $i$ and every $m$, this path must have infinitely many edges. Hence $\Supp{h}$ must be infinite.
\end{proof}

We now use Theorem \ref{thm:Magnus conjugators}, Lemma \ref{lem:image not (1,b)} and work from Section \ref{sec:conjugacy in wreath products} to give an estimate of the conjugacy length function of a group $F/N'$ with respect to the subgroup distortion of its cyclic subgroups. Recall that \mbox{$\bar{\alpha}:F/N' \rightarrow F/N$} denotes the canonical homomorphism and $\delta_{\langle \bar{\alpha}(u) \rangle}^{F/N}$ is the distortion function for the subgroup of $F/N$ generated by $\bar{\alpha}(u)$.

\begin{thm}\label{thm:F/N' conj length}
Let $N$ be a normal subgroup of $F$ such that $F/N$ is torsion-free. Let $u,v$ be elements in $F/N'$. Then $u,v$ are conjugate in $F/N'$ if and only if there exists a conjugator $\gamma \in F/N'$ such that
	$$d_{F/N'}(1,\gamma) \leq  (16n^2 +8n)(2\delta_{\langle \bar{\alpha}(u) \rangle}^{F/N}(4n)+1).$$
In particular,
	$$\CLF_{F/N'}(n) \leq (16n^2+ 8n)(2\Delta_{\mathrm{cyc}}^{F/N}(4n) +1)$$
where $\Delta_{\mathrm{cyc}}^{F/N}(m)=\sup \left\{ \delta_{\langle x \rangle}^{F/N}(m) \mid x \in F/N\right\}$.
\end{thm}

\begin{proof}
We begin by choosing a conjugator $(h,\gamma) \in \Z^r \wr F/N$ for which $\gamma$ is short, as in Lemma \ref{lem:wreath conjugator with short B part}. Then Theorem \ref{thm:Magnus conjugators} tells us that there exists some lift $\gamma_0$ of $\gamma$ in $F/N'$ such that $\varphi(\gamma_0)=(h_0,\gamma)$ is a conjugator. First suppose that $u \notin N$. Then $\overline{\alpha}(u)$ has infinite order in $F/N$ and we may apply Lemma \ref{lem:wreath conj bound depending on B part} to give us
				$$d_M(1,\varphi(\gamma_0)) \leq  (n'+1)P'(2\delta_{\langle \bar{\alpha}(u) \rangle}^{F/N}(P')+1)$$
where $n' = d_M(1,\varphi(u))+d_M(1,\varphi(v))$ and $P'=2n'$ if $\varphi(u)$ is not conjugate to $(1,\bar{\alpha}(u))$ or $P'=n'+\CLF_{F/N}(n')$ otherwise. Lemma \ref{lem:image not (1,b)} tells us that we can discount the latter situation and set $P'=2n'$. By Theorem \ref{thm:F/N' undistorted in M(F/N)} we see that $d_{F/N'}(1,\gamma) \leq 2 d_{M}(1,\gamma_0)$ and $n' \leq 2n$. Then, provided $u \notin N$, the result follows.

If, on the other hand, $u \in N$, then we must apply the torsion version, Lemma \ref{lem:wreath torsion bound depending on B part}. The upper bound we get on $d_M(1,\varphi(\gamma_0))$ this time will be $P'(2n'+1)$, where $p'=2n'$ and $n'\leq 2n$ as in the previous case. This leads to
				$$d_{F/N'}(1,\gamma_0)\leq 16n^2+4n.$$
Clearly this upper bound suffices to give the stated result.
\end{proof}

In the special case where $N'=F^{(d)}$ the quotient $F/N'$ is the free solvable group $S_{r,d}$ of rank $r$ and derived length $d$. Plugging Proposition \ref{prop:cyclic subgroup distortion in free solvable groups} into Theorem \ref{thm:F/N' conj length} gives us an upper bound for the length of short conjugators between two elements in free solvable groups.

\begin{cor}\label{cor:CLF of free solvable}
Let $r,d > 1$. Then the conjugacy length function of $S_{r,d}$ is a cubic polynomial.
\end{cor}

\begin{proof}
In light of Proposition \ref{prop:cyclic subgroup distortion in free solvable groups}, applying Theorem \ref{thm:F/N' conj length} gives us a conjugator $\gamma \in S_{r,d}$ such that $d_{S_{r,d}}(1,\gamma) \leq (16n^2 +8n)(16n+1).$
\end{proof}

We may ask whether this upper bound is sharp. Indeed, Theorem \ref{thm:Z^r wr Z^r lower bound on clf} tells us it is possible to find elements in $\Z^r \wr \Z^r$ which observe a quadratic conjugacy length relationship. However it seems that this will not necessarily carry through to the free metabelian groups $S_{r,2}$ as the elements $u_n$ and $v_n$ considered in Theorem \ref{thm:Z^r wr Z^r lower bound on clf} cannot be recognised in the image of the Magnus embedding for $S_{r,2}$. Restricting to elements in this image seems to place too many restrictions on the nature of the support of the corresponding functions of the conjugate elements. It therefore seems plausible that the conjugacy length function for $S_{r,2}$ could be subquadratic.

\bibliography{bibliography}{}
\bibliographystyle{alpha}

\end{document}